\documentclass[letterpaper, 11pt,  reqno]{amsart}
\usepackage[margin=1.2in,marginparwidth=1.5cm, marginparsep=0.5cm]{geometry}

\usepackage{pdfpages,caption,geometry}
\usepackage[most]{tcolorbox}

\usepackage{mdframed}
\usepackage{enumitem}
\usepackage{amsmath,amssymb,amscd,amsthm,amsxtra}
\usepackage[implicit=true]{hyperref}
\usepackage{bbm}

\usepackage{comment}

\usepackage{stackengine}

\usepackage[textwidth=25mm]{todonotes}

\newtheorem{theorem}{Theorem}[section]
\newtheorem{lemma}[theorem]{Lemma}
\newtheorem{proposition}[theorem]{Proposition}
\newtheorem{remark}[theorem]{Remark}

\newtheorem*{ack}{Acknowledgments}

\DeclareMathOperator*{\supp}{supp}

\usepackage{tikz}
\usetikzlibrary{matrix} 

\usepackage{enumitem}

\allowdisplaybreaks

\usepackage{color}

\renewcommand{\H}{\mathcal{H}}

\newcommand{\noi}{\noindent}
\newcommand{\Z}{\mathbb{Z}}
\newcommand{\R}{\mathbb{R}}

\newcommand{\T}{\mathbb{T}}

\newcommand{\D}{\mathcal{D}}

\let\P= \undefined
\newcommand{\P}{\mathbf{P}}

\newcommand{\E}{\mathbb{E}}

\renewcommand{\L}{\mathcal{L}}

\newcommand{\NB}{\mathbb{N}}

\newcommand{\F}{\mathcal{F}}

\def\norm#1{\|#1\|}

\newcommand{\al}{\alpha}
\newcommand{\be}{\beta}
\newcommand{\dl}{\delta}
\newcommand{\nb}{\nabla}

\newcommand{\Dl}{\Delta}

\newcommand{\eps}{\varepsilon}

\newcommand{\g}{\gamma}
\newcommand{\G}{\Gamma}
\newcommand{\ld}{\lambda}
\newcommand{\Ld}{\Lambda}
\newcommand{\s}{\sigma}

\newcommand{\ft}{\widehat}

\newcommand{\wt}{\widetilde}
\newcommand{\wb}{\overline}
\newcommand{\cj}{\overline}

\newcommand{\dt}{\partial_t}

\newcommand{\ta}{\theta}

\renewcommand{\l}{\ell}
\renewcommand{\o}{\omega}
\renewcommand{\O}{\Omega}

\newcommand{\les}{\lesssim}
\newcommand{\ges}{\gtrsim}

\newcommand{\jb}[1]
{\langle #1 \rangle}

\newcommand{\jbb}[1]
{[\hspace{-0.6mm}[ #1 ]\hspace{-0.6mm}]}

\newcommand{\Ws}[1]
{\mathcal{W}_\sigma (#1) }

\newcommand{\Wa}[1]
{\mathcal{W}_\alpha (#1) }

\newcommand{\W}[1]
{\mathcal{W} (#1) }

\newcommand{\rhoo}{\vec{\rho}}
\newcommand{\muu}{\vec{\mu}}
\newcommand{\deff}{\stackrel{\textup{def}}{=}}

\newcommand{\Psid}{{\Psi}^{\mathrm d}}

\usepackage{bbm}

\DeclareMathOperator{\Id}{Id}

\numberwithin{equation}{section}
\numberwithin{theorem}{section}

\usepackage{todonotes}

\title[Global dynamics for 4-$d$ stochastic NLB]
{Global dynamics for the stochastic nonlinear beam equations
on the four-dimensional torus}
\author{Andreia Chapouto, Guopeng Li, and Ruoyuan Liu}

\begin{document}
\baselineskip = 14pt


%

\address{
Andreia Chapouto\\ Maxwell Institute for Mathematical Sciences
and 
School of Mathematics\\
The University of Edinburgh\\
and The Maxwell Institute for the Mathematical Sciences\\
James Clerk Maxwell Building\\
The King's Buildings\\
Peter Guthrie Tait Road\\
Edinburgh\\ 
EH9 3FD\\United Kingdom} 

\email{a.chapouto@ed.ac.uk}

\address{
Guopeng Li\\ Maxwell Institute for Mathematical Sciences
and 
School of Mathematics\\
The University of Edinburgh\\
and The Maxwell Institute for the Mathematical Sciences\\
James Clerk Maxwell Building\\
The King's Buildings\\
Peter Guthrie Tait Road\\
Edinburgh\\ 
EH9 3FD\\United Kingdom} 

\email{guopeng.li@ed.ac.uk}

\address{
Ruoyuan Liu\\ Maxwell Institute for Mathematical Sciences
and 
School of Mathematics\\
The University of Edinburgh\\
and The Maxwell Institute for the Mathematical Sciences\\
James Clerk Maxwell Building\\
The King's Buildings\\
Peter Guthrie Tait Road\\
Edinburgh\\ 
EH9 3FD\\United Kingdom} 

\email{ruoyuan.liu@ed.ac.uk}




\begin{abstract}
We study global-in-time dynamics 
of the stochastic nonlinear beam equations (SNLB) with 
an additive space-time white noise, 
posed on the four-dimensional torus.
The roughness of the noise leads us to introducing a time-dependent renormalization, after which we show that SNLB is pathwise locally well-posed in all subcritical and most of the critical regimes. For the (renormalized) defocusing cubic SNLB, we establish pathwise global well-posedness
below the energy space, by adapting a hybrid argument of Gubinelli-Koch-Oh-Tolomeo (2022) that combines the $I$-method with a Gronwall-type argument. Lastly, we show almost sure global well-posedness and invariance of the Gibbs measure for the stochastic damped nonlinear beam equations in the defocusing case.

\end{abstract}



%
\maketitle

\section{Introduction}\label{SEC:intro}
We consider the stochastic nonlinear beam equation (SNLB) on $\T^4 = (\R/\Z)^4$ with additive space-time white noise:
\begin{equation}\label{SNLB}
	\begin{cases}
		\dt^2 u + \Dl^2 u \pm u^k = \xi , \\
		(u, \dt u)\vert_{t=0} = (u_0, u_1) \in \mathcal{H}^s(\T^4),
	\end{cases}
\end{equation}
where $u: \R_+\times\T^4 \to \R$, 
$\Delta^2$ denotes the bi-harmonic operator, $k\ge2$ is a natural number, $\xi$ is a space-time white-noise on $\R_+\times \T^4$, and $\H^s(\T^4) = H^s(\T^4) \times H^{s-2}(\T^4)$. We refer to the equation \eqref{SNLB} with `$+$' as defocusing and with `$-$' as focusing.

The deterministic beam equation appears in the literature under various names, such as the fourth-order wave equation, the extensible beam/plate equation, and the Bretherton equation. In the one-dimensional setting, it was first derived by Bretherton in \cite{Br64} to describe the weak interaction between dispersive waves and it has a variety of applications in physics and mechanics; see \cite{PT01} and references therein. We also refer to the non-local model derived by Woinowsky-Krieger \cite{SWK50} to describe the vibration of a clamped extensible beam. 

Our main goal is to establish low regularity well-posedness of \eqref{SNLB} on $\T^4$ with space-time white-noise, which is of analytical interest due to the roughness of the noise. This study on $\T^d$ for $d=1,2,3$ was pursued in \cite{MPTW, Tolomeo-thesis, Tolomeo20}. We also mention the results in \cite{BMS05, Chow06, BOS15} for a non-local version of \eqref{SNLB} with multiplicative noises. For the study of the deterministic nonlinear beam equation, we refer the interested readers to \cite{Lebeau92, P07, P10} and references therein.

The main difficulty in studying \eqref{SNLB} on $\T^4$ comes from the roughness of the noise $\xi$. 
To illustrate this, we first consider the mild formulation of \eqref{SNLB}:
\begin{equation}\label{mild}
	u(t) = S(t)(u_0,u_1) \mp \int_0^t \frac{\sin((t-t')\Dl)}{\Dl} u^k(t')\,dt' + \Psi(t),	
\end{equation}
where $S(t)$ denotes the linear propagator
\begin{equation}\label{defS}
S(t)(u_0,u_1) = \cos(t\Dl) u_0 + \frac{\sin(t\Dl)}{\Dl} u_1,
\end{equation}
with the understanding that $\frac{\sin(t \cdot 0)}{0} = t $, and $\Psi$ is the stochastic convolution which solves the linear stochastic beam equation on $\T^4$:
\begin{equation}
	\dt^2 \Psi + \Dl^2 \Psi   =  \xi.
	\label{SNLB1b}
\end{equation}
More precisely, $\Psi$ is given by
\begin{align}
	\Psi (t) = \int_0^t \frac{\sin ((t - t') \Dl)}{\Dl} dW (t'),
	\label{defPsi}
\end{align}
where $W$ denotes a cylindrical Wiener process on $L^2 (\T^4)$:
\begin{align}
	W(t,x) \deff \sum_{n\in \Z^{4} }  \beta_n(t) e_n(x)
	\label{Wiener1}
\end{align} 
\noi
with\footnote{Here and after, we drop the harmless factor of $2\pi$.}
$
e_n(x) = e^{2\pi  in\cdot x}
$
and  
$\{\be_n\}_{n \in \Z^4 }$ is a family of mutually independent complex-valued Brownian motions conditioned to $\be_{-n} = \cj{\be_n}$, $n\in\Z^4$, with variance $\operatorname{Var}(\be_n(t))=t$. One can show that $W$ lies almost surely in\footnote{In general, we have
	$W\in C^{\frac12-}(\R_+;W^{-\frac{d}{2}-,\infty})$ for $d\geq1$, which follows by Kolmogorov’s continuity criterion and \cite[Lemma 2.6]{GKO2}. Here $W^{s,r}(\T^4)$ denotes the usual $L^r$-based Sobolev spaces defined via the norm in \eqref{Wsp}.}
$C^{\al}(\R_+;H^{-2-\eps}(\T^4))$
for any $\al < \frac 12$ and $\eps > 0$. Therefore, due to the two degrees of spatial smoothing of the linear beam equation, it follows that $\Psi(t) \in H^{-\eps}(\T^4) \setminus L^2(\T^4)$ almost surely, for any $\eps>0$, thus it is merely a distribution. 
Consequently, we expect the solution $u$ to \eqref{mild} to also only be a distribution and thus the product $u^k$ is classically ill-defined. To overcome this difficulty, we closely follow the work of Gubinelli-Koch-Oh \cite{GKO}  for wave equations (see also \cite{OT2}), and construct solutions $u=\Psi+v$ which solve a suitably renormalized version of \eqref{SNLB}.

We now detail this renormalization procedure. We first smooth the noise $\xi$ in \eqref{SNLB} via Fourier truncation and consider the truncated stochastic convolution $\Psi_N$ given by  
\begin{align*} 
	\Psi_N(t, x)  =   \pi_N \Psi(t, x)
	= \sum_{\substack{n \in \mathbb{Z}^4 \\ |n| \leq N}} e_n(x)
	\int_0^t \frac{\sin ((t - t') |n|^2)}{ |n|^2  } d  \beta_n (t'),
\end{align*}
where $\pi_N$ denotes the frequency truncation onto $\{|n| \le N\}$. 
Then, for each fixed $x \in \T^4$ and $t \geq 0$,  
it follows from the Ito isometry
that the random variable $\Psi_N(t, x)$
is a mean-zero real-valued Gaussian random variable with variance
\begin{align}
	\s_N(t) &= \E \big[\Psi_N^2(t, x)\big]
	\sim  t \log N,
	\label{sig1}
\end{align}

\noi
which is independent of $x \in \T^4$.

Let $u_N$ be the solution to SNLB \eqref{SNLB} with the regularized noise $\pi_N \xi$, which satisfies the mild formulation \eqref{mild} with the truncated stochastic convolution $\Psi_N$. Motivated by \eqref{mild}, we introduce the first order expansion (\cite{McKean, BO96, DPD}):
\begin{equation}
	\label{eq:ansatz}
	u_N = \Psi_N + v_N, 
\end{equation}
where the remainder $v_N$ solves the following nonlinear beam equation:
\begin{equation}
	\label{eq:vN}
	\dt^2 v_N + \Dl^2 v_N \pm \sum_{\ell=0}^k {k\choose \ell} \Psi_N^\ell v_N^{k-\ell} = 0.
\end{equation}
Unfortunately, due to \eqref{sig1}, the monomials $\Psi_N^\l$ in \eqref{eq:vN} do not have good limiting behavior as $N\to\infty$. Instead, we define the \emph{Wick-ordered} power $\Ws{  \Psi^\l_N  }$
as
\begin{align}
	\Ws{ \Psi^\l_N(t, x) }   \, \stackrel{\text{def}}{=} H_\ell(\Psi_N(t, x);\sigma_N(t)),
	\label{Herm1}
\end{align}

\noi
where $H_\l(x; \s )$ is the Hermite polynomial of degree $\l$, which can be shown to converge to a limit $\Ws{\Psi^\l}$ in $L^p(\O; C([0,T];W^{-\eps,\infty}(\T^4)))$,
for any $1\le p < \infty$ and $\eps > 0$ as $N \to \infty$; 
see Subsection \ref{SUBSEC:2a}. 
We then consider the Wick renormalized version of \eqref{eq:vN}
\begin{align}
	\label{eq:vN_wick}
	\dt^2 v_N + \Dl^2 v_N \pm \sum_{\ell=0}^k {k\choose \ell} \Ws{\Psi_N^\ell} v_N^{k-\ell} = 0,
\end{align}
which
converges, as $N\to\infty$, to the following equation:
\begin{align}
	\label{eq:v_wick}
	\dt^2 v + \Dl^2 v \pm \sum_{\ell=0}^k {k\choose \ell} \Ws{\Psi^\l} v^{k - \l} = 0.
\end{align} 
Lastly, from \eqref{eq:ansatz} and
\eqref{Herm3-1} below, 
we can define the Wick-ordered nonlinearity $\Ws{u_N^k}$ as
\begin{equation*}
	\Ws{u_N^k (t, x)} \deff H_k(u_N(t, x);\s_N(t)) = \sum_{\l=0}^k {k\choose \ell}  \Ws{ \Psi_N^\l (t,x) }
	v_N^{k-\ell} (t, x).
\end{equation*}
Consequently, if $v_N$ solves \eqref{eq:vN_wick}, then $u_N = \Psi_N + v_N$ satisfies the following truncated Wick renormalized SNLB:
\begin{align}
	\dt^2 u_N + \Dl^2 u_N \pm \Ws{u_N^k} = \pi_N \xi.
	\label{SNLB2}
\end{align} 
Similarly, with $u=\Psi + v$ for some suitable $v$, we define the Wick-ordered nonlinearity~as 
\begin{align}
	\label{eq:wick}
	\Ws{  u^k}  \deff  \sum_{\ell=0}^k {k\choose \ell} \Ws{ \Psi^\l   } v^{k - \l} ,
\end{align}
and so if $v$ solves \eqref{eq:v_wick}, then $u = \Psi + v$ 
solves the following Wick renormalized SNLB:
\begin{align}
	\dt^2 u + \Dl^2 u \pm \Ws{u^k} = \xi.
	\label{SNLB3}
\end{align}

\medskip

Before stating our first main result on local well-posedness of \eqref{SNLB3}, let us discuss the scaling critical regularity associated to  the deterministic nonlinear beam equation (NLB):
\begin{align}
	\label{NLB1}
	\dt^2 u +  \Dl^2   u   \pm  u^k = 0.
\end{align}
On $\R^4$, \eqref{NLB1} enjoys the following scaling symmetry: if $u$ is a solution to~\eqref{NLB1} then $u_\ld(t,x) \deff \ld^{\frac{4}{k-1}} u(\ld^2 t, \ld x)$ is also a solution to \eqref{NLB1}. This induces the scaling critical Sobolev index
$s_\text{scaling} = 2 - \frac{4}{k-1}$, i.e., the homogeneous Sobolev $\dot{H}^{s}(\R^4)$-norm with $s=s_{\text{scaling}}$ is invariant under the scaling. Moreover, for a given integer $k \ge 2$, we define $s_\text{crit}$ by 
\begin{align}\label{eq_scrit}
	s_\text{crit} \deff \max(s_\text{scaling}, 0) = \max\bigg(2 - \frac{4}{k-1},  0\bigg),
\end{align}
where the restriction $s_{\text{crit}} \ge0$ appears in making sense of the powers of $u$. Although the scaling symmetry does not extend to $\T^4$, 
the numerology still plays an important role in predicting local well-posedness issues. In particular, our aim is to show that the 
SNLB \eqref{SNLB} is locally well-posed in the scaling (sub)critical Sobolev spaces $\H^s(\T^4)$ with $s \geq s_{\text{crit}}$. In fact, we show pathwise local well-posedness of \eqref{SNLB3} in the subcritical regime for $s>s_{\text{crit}}$ and all order nonlinearities $k\ge2$, and also in the critical case ($s=s_{\text{crit}}$) for $k\ge4$.

\begin{theorem}\label{THM:1}
	Given  an integer  $k \geq 2$, let $s_\textup{crit}$ be as in \eqref{eq_scrit}.
	Then, the Wick renormalized SNLB \eqref{SNLB3}
	is pathwise locally well-posed in $\H^s(\T^4)$
	for 
	\[ \textup{(i) } k \geq 4: 
	\ s \geq  s_\textup{crit}
	\qquad \text{or} \qquad 
	\textup{(ii) }  k = 2, 3: 
	\ s >  s_\textup{crit}.
	\]

	\noi
	More precisely, 
	given any $(u_0, u_1) \in \H^s(\T^4)$, 
	there exists an almost surely positive stopping time $T =  T( \o, u_0, u_1)>0$
	such that there exists  a unique solution $u $ to  \eqref{SNLB3} on $[0, T]$
	with $(u, \dt u)\vert_{t=0} = (u_0,u_1)$ and
	\begin{align*}
		u \in    \Psi + C([0, T]; H^{s'}(\T^4)) 
		\subset C([0, T]; H^{-\eps}(\T^4)) 
	\end{align*}
	
	\noi
	for any $\eps > 0$, where $s' = \min(s, 2-\eps)$.
\end{theorem} 

The solution $u$ in Theorem~\ref{THM:1} is understood as $u=\Psi +v$ where we construct $(v,\dt v) \in C([0,T]; \H^{s'}(\T^4))$ with $v$ solving the following Duhamel formulation:
\begin{align}
	v(t)= 
	S(t) (u_0, u_1)
	\mp \int_{0}^t  \frac{\sin ((t-t') \Dl )}{ \Dl}  \Ws {u^k(t')} dt',
	\label{SNLB9}
\end{align}
for $\Ws{u^k}$ and $S(t)$ as in \eqref{eq:wick} and \eqref{defS}, respectively.
The main ingredient in proving Theorem~\ref{THM:1} in the (almost) critical regime are the Strichartz estimates for the beam equation. 
In the Euclidean setting, by exploiting the formal decomposition 
$$\dt^2 +\Dl^2 = (i\dt + \Dl)(-i\dt +\Dl), $$
which sheds light on the relation between the beam equation and the Schr\"odinger equation, and the analysis of oscillatory integrals, Pausader \cite{P07, P10} established Strichartz estimates for the beam equation. However, in contrast to the wave equation, the lack of finite speed of propagation poses difficulties in transferring these estimates from the Euclidean to the periodic setting. Instead, we exploit the connection between the operator $S(t)$ in \eqref{defS} appearing in \eqref{SNLB9} and the free Schr\"odinger operators $e^{\pm i t\Dl}$ via the periodic Schr\"odinger Strichartz estimates in \cite{BD15, KV16} from the $\l^2$-decoupling theory. See Section~\ref{SEC:LWP} for details.

{

\begin{remark}
	\rm 
	(i) In Theorem~\ref{THM:1}, we cannot reach the critical regularity $s=s_{\text{crit}}=0$ for the quadratic and cubic SNLB \eqref{SNLB3}, $k=2,3$. This restriction comes from the sharp Strichartz estimates for Schr\"odinger (see Lemma~\ref{LEM:decoupling}), where the endpoint $p=3$ is not included, which is needed for our argument in the critical setting for $k=2,3$. 
	Strichartz estimates for $p=3$ are known to only hold with a derivative loss \cite{BO93, BD15}, which prevents us from taking $s=0$. Thus our result is sharp with respect to the method. It may be possible to reach the critical regularity in these cases by using the $U^p$-$V^p$ spaces introduced in \cite{KT05}.

	\smallskip
	
	\noi (ii) \label{REM:thm1} The proof of Theorem~\ref{THM:1} can be easily adapted to show local well-posedness of the truncated Wick-ordered SNLB \eqref{SNLB2}, uniformly in $N$. In fact, it follows that for $(u_0,u_1)\in \H^s(\T^4)$, there exists an almost surely positive stopping time $T=T(\o, u_0,u_1)>0$ independent of $N$ and a unique solution
	$u_N \in \Psi_N + C([0,T]; H^{s'}(\T^4))$ to \eqref{SNLB2}. Moreover, we can show that $u_N$ converges to the corresponding solution $u$ to \eqref{SNLB3}. 
We note that although this seems to depend on regularizing by $\pi_N$, one can consider a different regularization procedure, such as mollification. Indeed, one can show that the Wick-ordered monomials $\Ws{\Psi^k}$ are independent of the choice of mollifier, and thus so is the renormalized nonlinearity \eqref{eq:wick}. See \cite[Remark~1.2]{GKO} for further discussion.

\end{remark}

}

\medskip

Our next goal is to extend the solutions constructed in Theorem~\ref{THM:1} globally-in-time. We restrict our attention to the defocusing case (`$+$' sign in \eqref{SNLB}) and odd-ordered nonlinearities, as the energies corresponding to the deterministic NLB equation are sign definite in this setting. First, we construct pathwise global-in-time solutions for the cubic defocusing Wick-ordered \eqref{SNLB3} by adapting the hybrid method of Gubinelli-Koch-Oh-Tolomeo \cite{GKOT} to the beam equation. Then, we use Bourgain's invariant measure argument to show almost sure global well-posedness and invariance of the Gibbs measure for the defocusing damped Wick renormalized SNLB with odd-power nonlinearities.

\smallskip 

We first consider the cubic Wick renormalized SNLB \eqref{SNLB3} in the defocusing case, with $k=3$ and `$+$' sign. In Theorem~\ref{THM:1}, for $s>0$ we constructed a solution $u=\Psi +v$ where the remainder $v$ solves
\begin{equation}\label{SNLB6}
	\dt^2 v + \Dl^2 v + \Ws{u^3} =0 , 
\end{equation}
and $\Ws{u^3}$ is given in \eqref{eq:wick}. 
	A consequence of the (deterministic) contraction argument used to show Theorem~\ref{THM:1} is the following (almost sure) blow-up alternative: either the solution $v$ exists globally in time
	or there exists some finite time $T_* = T_*(\o) > 0$
	such that  
	\begin{align}
		\lim_{t \nearrow T_*}\| \vec v(t) \|_{\H^{s'}} = \infty,
		\label{BA}
	\end{align}
	
	\noi
	where $\vec v = (v, \dt v)$
	and $s' = \min(s, 2- \eps)$ for any small $\eps > 0$. 
	
	To globalize solutions, we must control the growth of the norm in \eqref{BA}. 
	In the parabolic setting, there are various results where deterministic arguments have been adapted to the stochastic setting to directly control the growth of norms of solutions; see \cite{MW1, MW2, GH, MoW}. Unfortunately, for \eqref{SNLB3}, due to the lack of a strong smoothing effect, such arguments do not apply. Instead, even in the deterministic setting, we must consider conservation laws.
	For the deterministic cubic nonlinear beam equation (NLB):
	\begin{equation*}
		\dt^2 v  + \Dl^2 v +  
		v^3  = 0,
	\end{equation*}
	the associated energy
	\begin{align}
		E(\vec v ) = \frac{1}{2}\int_{\T^4} (\Dl v)^2 dx
		+ 
		\frac{1}{2}\int_{\T^4} (\dt v)^2dx
		+ \frac1{4} \int_{\T^4} v^{4} dx,
		\label{Hamil}
	\end{align}
	gives control over the $\H^2(\T^4)$-norm of $\vec v$, as this quantity is conserved for sufficiently regular solutions. 
	Unfortunately, when adding noise to the equation and considering a solution $v$ to \eqref{SNLB6}, two problems arise: (i) the energy $E(\vec v)$ is not conserved under the dynamics of \eqref{SNLB6}, and (ii) since $\vec v\in \H^{s'}(\T^4) \setminus \H^{2}(\T^4)$ for $s'=\min(s,2-\eps)$ for any $\eps>0$, the energy $E(\vec{v})$ is actually infinite. 
	
	In the context of the two-dimensional cubic stochastic nonlinear wave equation, Gubinelli-Koch-Oh-Tolomeo \cite{GKOT} introduced a new hybrid method to overcome these difficulties, by combining the $I$-method of Colliander-Keel-Staffilani-Takaoka-Tao~\cite{CKSTT1, CKSTT2} and the Gronwall-type globalization argument by Burq-Tzvetkov~\cite{BT2}. See also \cite{Forlano20, Tolomeo21} for other instances of this method. To establish our next main result, we adapt this argument to show pathwise global well-posedness of \eqref{SNLB6}.

\begin{theorem}\label{THM:GWP1}
	Let $s > \frac 74$.
	Then, the defocusing cubic Wick renormalized SNLB \eqref{SNLB6}
	is globally well-posed in $\H^s(\T^4)$.
	More precisely, 
	given any $(u_0, u_1) \in \H^s(\T^4)$, 
	the solution $v$ to the equation \eqref{SNLB6} 
	exists globally in time and
	$(v, \dt v) \in C(\R_+; \H^{s'}(\T^4))$, almost surely, for
	$s' = \min(s, 2- \eps)$ for any small $\eps > 0$.
\end{theorem}

We briefly detail the ideas of the proof of Theorem~\ref{THM:GWP1}. For simplicity, let $\frac74<s<2$ so that $s'=s$. In view of the blow-up alternative \eqref{BA}, our main goal is to control the $H^s(\T^4)$-norm of the solution $v$ to \eqref{SNLB6}, where the conservation of $E(\vec v)$ is not useful. Instead, the $I$-method is based on studying the growth of a modified energy obtained from $E(\vec v)$ which controls the $H^s(\T^4)$-norm of $v$. In particular, for $N\in\NB$, we consider $E (I \vec{v})$ where $I=I_N$ denotes the $I$-operator, a Fourier operator with a smooth, radially symmetric, non-increasing multiplier $m_N$ given by
\begin{equation}
	m_N(\xi)=
	\begin{cases}
		1, & \text{if } |\xi|\le N, \\
		\Big(\frac{N}{|\xi|}\Big)^{2-s}, & \text{if } |\xi|\ge2N.
	\end{cases}
	\label{I0a}
\end{equation} 
Note that $If \in H^2(\T^4)$ if and only if $f \in H^s(\T^4)$; see \eqref{I1}-\eqref{I2}.

To study the growth of $E(I \vec v)$, we consider the following $I$-SNLB:
\begin{equation*}
	\dt^2 I v  + \Dl^2 I v +  
	I \Ws{u^3}= 0,
\end{equation*}
where $\Ws{u^3}$ is as in \eqref{eq:wick}.
Unfortunately, the modified energy $E(I\vec v)$ is not conserved under the flow of $I$-SNLB, 
 and by direct computation we obtain
\begin{equation}
\begin{aligned}
	E(I\vec v)(t_2) -  E(I\vec v)(t_1) 
	&  = \int_{t_1}^{t_2} \int_{\T^4} (\dt I v) \big\{ - I (v^3) + (Iv)^3\big\} dx dt'  \\
	& \hphantom{X}
	-3 
	\int_{t_1}^{t_2}\int_{\T^4} (\dt I v)  \big\{ I(v^2 \Psi)  +    I(v \Ws{  \Psi^2  })  \big\}  dx dt'  \\
	& \hphantom{X}
	- 
	\int_{t_1}^{t_2} \int_{\T^4} (\dt I v) I(\Ws{  \Psi^3  } ) dx dt' ,
\end{aligned}\label{E1}
\end{equation}
for $0\le t_1<t_2$.
The first term, due to the $I$-operator, requires a certain (deterministic) commutator estimate; see Lemma~\ref{LEM:C1}. The difficulty in the remaining contributions comes from the roughness of $\Psi$, which is handled by exploiting a finer regularity property of $I\Psi$ combined with commutator estimates and a Gronwall-type argument. Finally, due to the growth of the modified energy $E(I_N \vec v)$, we iterate the argument above over time-intervals of fixed size, but with an increasing sequence $N_k$ of parameters for the $I$-operator, extending the solution to \eqref{SNLB6} globally-in-time. See Section \ref{SEC:GWP1} for details.

\begin{remark}\rm
	
(i) There is a gap between the global well-posedness result for the Wick-ordered cubic SNLB \eqref{SNLB6} in Theorem \ref{THM:GWP1} for $s> \frac74$ and the local well-posedness threshold $s > 0$ from Theorem \ref{THM:1}. 
It may be possible to refine the $I$-method part of the argument to lower the regularity assumptions. However, we do not pursue this issue in this paper.

\smallskip

\noi (ii) At this point, we do not know how to prove pathwise global well-posedness of the Wick ordered defocusing SNLB \eqref{SNLB3} with a (super-)quintic nonlinearity. This may be possible for a smoother noise, by using ideas in \cite{OP16, Latocca21} to deal with the high degree of nonlinearity. See also \cite{MPTW} for SNLB on $\T^3$, where the authors proved pathwise global well-posedness in a super-cubic setting.

\smallskip

\noi(iii)
	A standard application of the $I$-method results in a polynomial growth bound (in time) on the Sobolev norm of a solution. 
	See, for example, \cite[Section 6]{CKSTT2}.
	The hybrid argument used for Theorem~\ref{THM:GWP1} yields
	a double exponential growth bound on the $\H^s$-norm of the solution; see Remark \ref{REM:bound} below. It may be possible to improve this double exponential bound, but we expect that one can obtain at best a polynomial growth bound for SNLB \eqref{SNLB6} due to the polynomial growth (in time) of the stochastic convolution $\Psi$. One can compare this situation with the damped case in the next subsection, where the invariant measure argument yields a logarithmic growth bound; see Remark \ref{RMK:inv}(i) below.

%

\end{remark}

\smallskip

Lastly, 
 we restrict our attention to the following (defocusing) stochastic damped nonlinear beam equation (SdNLB):
\begin{align}
	\dt^2 u + \dt u + (1 -  \Dl)^2  u   + u^k = \sqrt 2\xi
	\label{SdNLB}
\end{align}

\noi
for $k \in 2 \NB + 1$. By modifying the proof of Theorem~\ref{THM:GWP1}, we can show global well-posedness for the damped dynamics \eqref{SdNLB} when $k=3$, after renormalization, but we do not know how to extend this deterministic argument to higher nonlinearities. Instead, we consider a probabilistic approach and establish almost sure global well-posedness of \eqref{SdNLB} and invariance of the Gibbs measure $\vec \rho$ via Bourgain's invariant measure argument \cite{BO94, BO96}, where $\vec \rho$ is formally given by
\begin{align}
	\text{``}d\rhoo (u,  \dt u   ) = Z^{-1}e^{-E(u,\dt u )}du\,d(\dt u)\text{"}.
	\label{Gibbs1}
\end{align}
Here, $E(u, \dt u)$ denotes the energy (or Hamiltonian) of the deterministic undamped defocusing NLB \eqref{NLB1}:
\begin{align} 
	E(u, \dt u ) = \frac{1}{2}\int_{\T^4} [( 1-  \Dl) u]^2 dx
	+ 
	\frac{1}{2}\int_{\T^4} (\dt u)^2dx
	+ \frac1{k+1} \int_{\T^4} u^{k+1} dx.
	\label{Hamil2}
\end{align}
%
%

We can understand SdNLB \eqref{SdNLB} as a superposition of the defocusing NLB dynamics~\eqref{NLB1}  and the 
Ornstein-Uhlenbeck dynamics 
(for the  component $\dt u$):
\begin{align*}
	\dt (\dt u) = - \dt u + 
	\sqrt 2 dW.
\end{align*}

\noi The latter leaves the Gibbs measure $\rhoo$ invariant, which is also expected to hold under the dynamics of NLB \eqref{NLB1} due to its Hamiltonian structure; see \cite{OST} and  \cite[Chapter 3]{XB16}. Therefore, we expect $\rhoo$ to also be invariant under SdNLB \eqref{SdNLB}.

Moreover, this invariance is also inferred from the stochastic quantization viewpoint. In fact, \eqref{SdNLB} is the so-called canonical stochastic quantization equation 
of the $\Phi^{k + 1}_4$-model; see~\cite{RSS}. We thus refer to \eqref{SdNLB} as the hyperbolic $\Phi_4^{k + 1}$-model, which is of importance in constructive quantum field theory. The invariance of the Gibbs measure is also related to other applications in physics such as the study of equilibrium states, couplings of fields, and scattering of particles; see \cite{SS79, AZ82, AG83, GK85, FMRS87, HT86,  AD21} and references therein. See also \cite{GKO, Tolomeo20, GKOT, ORTz, LTW} for further results on wave-like $\Phi_d^{k+1}$-models.

Our first step is to rigorously construct the measure $\rhoo$, since the expression in \eqref{Gibbs1} is only formal. We want to define $\rhoo$ as a weighted Gaussian measure of the form
\begin{align}\label{Gibbs2}
	\text{``}d\rhoo  (u, \dt u) = Z^{-1} e^{
		- \frac1{k+1} \int_{\T^4} u^{k+1} dx} 
	d\muu_2 (u, \dt u)\text{''},
\end{align} 
where $\muu_2 = \mu_2 \otimes \mu_0$ and $\mu_s$ denotes a Gaussian measure on periodic distributions given by
\begin{align}
	d \mu_s 
	= Z_s^{-1} e^{-\frac 12 \| u\|_{{H}^{s}}^2} du
	& =  Z_s^{-1} \prod_{n \in \Z^4} 
	e^{-\frac 12 \jb{n}^{2s} |\ft u(n)|^2}   
	d\ft u(n),
	\label{gauss0}
\end{align}
for $s\in\R$. Note that $\mu_0$ corresponds to the white noise measure. More precisely, $\muu_2$ is defined as the induced probability measure
under the map $
	\o \in \O \longmapsto 
	(X^1(\o), X^2(\o)),$
where $X^1(\o)$ and $X^2(\o)$ are given by
\begin{equation}\label{series}
	X^1(\o) = \sum_{n \in \Z^4} \frac{g_n(\o)}{\jb{n}^2}e_n
	\qquad\text{and}\qquad
	X^2(\o) = \sum_{n \in \Z^4} h_n(\o)e_n.
\end{equation}
\noi
Here, 
$\{g_n,h_n\}_{n\in\Z^4}$ denotes  a family of independent standard complex-valued  Gaussian random variables conditioned so that $\cj{g_n}=g_{-n}$ and $\cj{h_n}=h_{-n}$, 
$n \in \Z^4$. The main difficulty in making sense of \eqref{Gibbs2} comes from the rough support of the base Gaussian measure $\muu_2$, namely $\H^{-\eps}(\T^4) \setminus \H^0 (\T^4)$
for any $\eps > 0$; see \cite[Lemma B.1]{BT1}. Since the typical element in the support of $\muu_2$ is merely a distribution, the term $\int_{\T^4} u^{k+1} \,dx $ in \eqref{Gibbs2} is ill-defined and a renormalization is needed in rigorously constructing $\rhoo$.

Similarly to the local theory for SNLB \eqref{SNLB}, where we introduced a renormalization based on the logarithmically diverging variance of $\Psi$ in \eqref{sig1}, here the same difficulty appears due to the roughness of the support of $\muu_2$. In fact, for $N\in\NB$, the typical element $X^1$ in the support of $\mu_2$ satisfies
\begin{align}\label{sN}
	\al_N \deff  \E\big[\big(\pi_N  X^1(x)\big)^2\big] 
	=\sum_{\substack{n\in\Z^4\\|n|\leq N}}\frac{1}{\jb{n}^4}
	\sim \log N ,
\end{align}
which is independent of both $t\in\R_+$ and $x\in\T^4$.
We then define the Wick renormalized truncated potential energy
	\begin{align}
	R_N(u) = 
	\frac{1}{k+1} 
	\int_{\T^4} \Wa{   (\pi_N u)^{k+1} } dx,
	\label{R1}
\end{align}

\noi
where the Wick-ordered power $\Wa{(\pi_N u)^{k+1}} $ is defined by  
\begin{align}
	\Wa{  (\pi_N u)^{k+1}(t,x) }
	\, \deff  H_{k+1} (\pi_N u(t,x); \al_N).
	\label{Herm2}
\end{align}

\noi
One can show that $\{R_N \}_{N \in \NB}$
forms a Cauchy sequence in $L^p(\mu_2)$
for any finite $p \geq 1$, from which we conclude that there exists a limiting random variable 
$R(u)$ given by
\begin{align}
	\lim_{N\to\infty} R_N(u) \deff R(u) = 
	\frac{1}{k+1} 
	\int_{\T^4} 
	\Wa{u^{k+1}(x) } dx.
	\label{R2}
\end{align}


\noi
See \cite[Proposition~1.1]{OT1} and \cite[Proposition~3.4]{LOZ} 
for details. We then construct the Gibbs measure $\rhoo$ as the limit of the following truncated Gibbs measures
\begin{align}\label{GibbsN}
	d\rhoo_N(u,\dt u )= Z_N^{-1}e^{-R_N(u)}d\muu_2(u,\dt u).
\end{align}

	\begin{proposition}
	\label{PROP:meas}
	Given any $1 \leq  p < \infty$, we have
	\begin{equation}
		\lim_{N\rightarrow\infty} e^{ -R_N(u)} = e^{-R(u)}
		\qquad \text{in } L^p( \mu_2).
		\label{Lp_conv}
	\end{equation}
	
	\noi
	Consequently, the truncated Gibbs measure $\rhoo_N$ in \eqref{GibbsN} converges, in the sense of \eqref{Lp_conv}, to a limiting Gibbs measure $\rhoo$ given by
	\begin{align}
		d\rhoo(u,\dt u) = Z^{-1} e^{-R(u)}d\muu_2(u,\dt u).
		\label{Gibbs3}
	\end{align}
\end{proposition}

We now sketch the proof of Proposition~\ref{PROP:meas}. From an application of Nelson's estimate, we obtain uniform in $N$ integrability of the truncated density; 
for any $1\le p <\infty$, 
\begin{equation}
	\sup_{N\in \NB} \big\| e^{-R_N(u)} \big\|_{L^p( \mu_2)} <\infty.
	\label{exp1}
\end{equation}
See, for example, \cite[Proposition 1.2]{OT1} and \cite[Proposition 3.6]{LOZ}. Combining the uniform bound \eqref{exp1} with a convergence in measure deduced from \eqref{R2}, we obtain \eqref{Lp_conv}; see, for example, \cite[Remark 3.8]{Tzv08} and \cite[(3.32)]{LOZ}.
	This allows us to construct the Gibbs measure $\rhoo$ in \eqref{Gibbs3}, which is mutually absolutely continuous with respect to the base Gaussian measure~$\muu_2$.

	We can now consider the dynamical problem for the 
 $\Phi^{k+1}_4$-model \eqref{SdNLB}. In particular, we consider the following truncated Wick renormalized SdNLB
	\begin{align}
		\dt^2 u_N   + \dt u_N  +(1-\Dl)^2  u_N 
		+
		\pi_N\big(    \Wa{ (\pi_N u)^{k} } \big) 
		= \sqrt{2} \xi ,
		\label{SNLB10}
	\end{align} 
	and 
	show almost sure global well-posedness and invariance of the Gibbs measure $\rhoo$ for the limiting equation:
	\begin{align}
		\dt^2 u   + \dt u  +(1-\Dl)^2  u 
		+
		\Wa{u^k} 
		= \sqrt{2} \xi  .
		\label{SNLB11}
	\end{align}

	\begin{theorem}\label{THM:2-old}
		Let $k \in 2 \NB + 1$. 
		 The Wick renormalized SdNLB~\eqref{SNLB11} is almost surely globally well-posed with respect to the Gibbs measure~$\rhoo$ in~\eqref{Gibbs3} and the Gibbs measure $\rhoo$ is invariant under the dynamics. More precisely, there exists a non-trivial stochastic process $(u,\dt u)\in C(\R_+;\H^{-\eps}(\T^4))$ for any $\eps>0$ such that, given any $T>0$, the solution $(u_N,\dt u_N)$ to 
		the renormalized truncated SdNLB~\eqref{SNLB10} with 
		random initial data  
		$(u_N, \dt u _N)|_{t = 0}$ distributed according 
		to the truncated Gibbs measure $\rhoo_N$  in~\eqref{GibbsN}, converges in probability to 
		some stochastic process $(u,\dt u)$ in $C([0,T];\H^{-\eps}(\T^4))$.
		Moreover,  the law of $(u(t),\dt u(t))$ is given by the renormalized Gibbs measure $\rhoo$ in \eqref{Gibbs3}
		for any $t\ge 0$.
	\end{theorem} 

	By using Bourgain's invariant measure argument, due to the convergence of $\rhoo_N$ to $\rhoo$, Theorem~\ref{THM:2-old} follows once we construct the limiting process $(u,\dt u)$ locally-in-time
	with a good approximation property for
	the solution $u_N$ to \eqref{SNLB10} and establish invariance of the truncated measures $\rhoo_N$ under \eqref{SNLB10}. The former follows from adapting the proof of Theorem~\ref{THM:1} to the damped models \eqref{SNLB10}-\eqref{SNLB11}, while the latter exploits the Hamiltonian structure of the truncated system \eqref{SNLB10}. See Section \ref{SEC:GWP2} for details.
	
	\begin{remark}\rm

		\noi(i) \label{RMK:inv}
		Let $(u, \dt u)$ be the limiting process constructed in Theorem \ref{THM:GWP1}. Then, as a consequence of Bourgain's invariant measure argument, one can obtain the following logarithmic growth bound (in time):
		\begin{align*}
			\| (u(t), \dt u (t)) \|_{\H^{- \eps}} \leq C(\o) \big( \log (1 + t) \big)^{\frac{k}{2}}
		\end{align*}
		
		\noi
		for any $t \geq 0$. For details, see~\cite{ORTz}.
		
		\smallskip
		
	\noi (ii) The local well-posedness in Theorem~\ref{THM:1} can be easily adapted to the Wick renormalized SNLB with damped massive linear part $\dt^2 u + \dt u + (1-\Dl)^2u$, which we detail in Section~\ref{SEC:GWP2}. We choose to consider the massive linear part $(1-\Dl)^2$ instead of $\Dl^2$ to avoid a problem at the zero-th frequency when constructing the Gibbs measure $\rhoo$, as in \cite{OT1, GKOT}.

	\smallskip
	
	\noi(iii) The Gaussian measure $\mu_2$ is the log-correlated Gaussian free field
	on $\T^4$ studied in \cite{OST}, and thus
	the SdNLB dynamics \eqref{SdNLB} are associated with this log-correlated Gibbs measure.
	Our construction of $\rhoo$ in~\eqref{Gibbs3} is
	valid for $k\in 2\NB+1$ and with a plus sign in front of the potential energy in \eqref{Hamil2}. However, in the case of a focusing quartic interaction (i.e.,~with a minus sign in front of the potential energy and $k=3$ in \eqref{Hamil2}), the authors in \cite{OST} obtained a non-normalizability result for the corresponding measure and established its exact divergence rate; see \cite[Theorem 1.4]{OST}.
	\end{remark}

\section{Preliminaries}
\label{SEC:2}

In this section, we introduce notations and recall basic lemmas.
For $a, b > 0$, we use $a \les b$ to denote that there exists a constant $C > 0$ such that $a \leq Cb$. We write $a \sim b$ if $a \les b$ and $b \les a$. When writing the norm of a space-time function, we usually use short-hand notation, such as $L_I^q L_x^r = L^q (I; L^r (\T^4))$ for a given time interval $I \subset \R_+$.

\subsection{Deterministic tools}
We first introduce some function spaces.
For $s \in \R$, we define the $L^2$-based Sobolev space $H^s(\T^4)$ via the norm:
\begin{align*}
\| f \|_{H^s}= \big\| \jb{n}^s \ft f (n) \big\|_{\l_n^2},
\end{align*}
where $\jb{\cdot}=(1+|\cdot|^2)^\frac12$ and $\ft{f}$ denotes the spatial Fourier transform of $f$.
For $1 \leq p \leq \infty$, we define the $L^p$-based Sobolev space $W^{s, p} (\T^4)$ via the norm:
\begin{align}\label{Wsp}
\| f \|_{W^{s, p}} = \big\| \F^{-1} \big( \jb{n}^s \ft f (n) \big) \big\|_{L^p},
\end{align}

\noi
where $\F^{-1}$ denotes the inverse Fourier transform. Note that $H^s (\T^4) = W^{s, 2} (\T^4)$.

We now introduce notation for Littlewood-Paley projections. Let $\phi : \R \to [0, 1]$ be a smooth bump function with $\supp\phi\subset[-\frac 85, \frac 85]$ and $\phi \equiv 1$ on $[-\frac 54, \frac 54]$. For $\xi \in \R^4$, we define
\begin{align*}
	\varphi_1 (\xi) = \phi(|\xi|) \qquad \text{and} \qquad
\varphi_{N} (\xi) = \phi \Big( \frac{|\xi|}{N} \Big) - \phi \Big( \frac{2|\xi|}{N} \Big)
\end{align*}

\noi
for $N \geq 2$ a dyadic number. For a dyadic number $N \geq 1$, we define the Littlewood-Paley projector $\P_N$ as the Fourier multiplier operator with the symbol $\varphi_N$. Then,
\begin{align*}
f = \sum_{N \geq 1 \text{ dyadic}} \P_N f.
\end{align*}

\noi
We also write
\begin{align*}
\P_{\leq N} f = \sum_{1 \leq M \leq N \text{ dyadic}} \P_M f.
\end{align*}

%

\medskip
Next, we recall the following Christ-Kiselev lemma. For a proof, see \cite{ChristKiselev, HassellTaoWunsch}.
\begin{lemma}
\label{LEM:CK}
	Let $X,Y$ be Banach spaces and $K(s,t): X \to Y$ be an operator-valued kernel from $X$ to $Y$. Suppose that we have the estimate
	$$\bigg\| \int_{-\infty}^{t_0} K(s,t) f(s) \, ds \bigg\|_{L^q([t_0, \infty); Y)}  \les \| f\|_{L^p(\R;X)},$$
	for some $ 1 \leq p < q \leq \infty$, all $t_0\in\R$, and all $f\in L^p\big((-\infty, t_0); X\big)$. Then, we have
	$$\bigg\| \int_{-\infty}^t K(s,t) f(s) \, ds \big\|_{L^q(\R; Y)} \les \|f\|_{L^p(\R;X)}. $$
\end{lemma}

	Note that the assumption in the above lemma is satisfied in particular if we have
	$$\bigg\| \int_\R K(s,t) f(s) \, ds \bigg\|_{L^q(\R;Y)} \les \|f\|_{L^p(\R;X)}.$$

\medskip
Lastly, we recall the following product estimates. See for example \cite[Lemma 3.4]{GKO}.
\begin{lemma}
\label{LEM:GKOlbnz}
	Let $0\leq s \leq 1$.
	
	\noi {\rm(i)} Suppose that $1<p_j,q_j,r < \infty$, $\frac{1}{p_j} + \frac{1}{q_j} = \frac{1}{r}$, $j=1,2$. Then, we have
		\begin{equation*}
			\| \jb{\nabla}^s (fg) \|_{L^r(\T^4)} \les \Big( \|f\|_{L^{p_1}(\T^4)} \|\jb{\nabla}^s g\|_{L^{q_1}(\T^4)} + \|\jb{\nabla}^s f \|_{L^{p_2}(\T^4)} \| g \|_{L^{q_2}(\T^4)} \Big).
		\end{equation*}
	
	\noi {\rm(ii)} Suppose that $1< p,q,r<\infty$ satisfy the scaling condition $\frac1p + \frac1q = \frac1r + \frac{s}{4}$. Then, we have
	\begin{equation*}
		\| \jb{\nabla}^{-s} (fg) \|_{L^r(\T^4)} \les \| \jb{\nabla}^{-s} f \|_{L^p(\T^4)} \| \jb{\nabla}^s g \|_{L^q(\T^4)}.
	\end{equation*}
\end{lemma}

\subsection{Tools from stochastic analysis}
\label{SUBSEC:2a}

In the following, we first review some basic facts on Hermite polynomials. 
See,  for example,  \cite{Kuo, Nualart06}.
We define the $k$-th Hermite polynomial
$H_k(x; \s)$ with variance $\s>0$ via the following generating function:
\begin{equation}
e^{tx - \frac{1}{2}\s t^2} = \sum_{k = 0}^\infty \frac{t^k}{k!} H_k(x;\s) 
\label{H1b}
 \end{equation}
 
\noi
for $t, x \in \R$.
We list the first few Hermite polynomials
for readers' convenience:
\begin{align*}
& H_0(x; \s) = 1, 
\qquad 
H_1(x; \s) = x, 
\qquad
H_2(x; \s) = x^2 - \s, \\
& H_3(x; \s) = x^3 - 3\s x, 
\qquad 
H_4(x; \s) = x^4 - 6\s x^2 +3\s^2.
\end{align*}

 \noi
From \eqref{H1b}, we obtain
the following identities for any $k \in \NB$ and $x,y\in\R$:
\begin{align}
H_k(x+y; \s )  &= 
\sum_{\l = 0}^k
\binom{k}{\l}
 x^{k - \l} H_\l(y; \s).
\label{Herm3-1}
\end{align}

We now recall the regularities of the stochastic convolutions and their Wick-powers introduced in Section~\ref{SEC:intro}. Let $\Psi$
be the stochastic convolution
defined in \eqref{defPsi} and $\Psid$ be the stochastic convolution associated with SdNLB \eqref{SdNLB}, namely the solution to the linear stochastic damped beam equation:
\begin{align}\label{SdLB}
	\begin{cases}
		\dt^2 \Psid + \dt\Psid +(1-\Dl)^2 \Psid  = \sqrt{2}\xi\\
		(\Psid,\dt\Psid)|_{t=0}=(u_0^\o, u_1^\o).
	\end{cases}
\end{align}
with initial data with law $\L(u_0^\o, u_1^\o) = \muu_2$. See Section~\ref{SEC:GWP2} for further details on $\Psid$.
Then, using standard stochastic analysis with the Wiener chaos estimate, 
we have the following regularity and convergence result.
For an analogous proof, we refer the readers to~\cite[Lemma 2.1]{LO22}. See also \cite{GKO, GKOT}.

\begin{lemma}
\label{LEM:sto_cov}

Let 
 $Z = \Psi$ or  $\Psid$, 
$\l \in \NB$, $T>0$, and $1\le p<\infty$.
For $ \W{ Z_N^\l }
\, = \, \W {(\pi_N Z)^\l }$
denoting the truncated Wick power defined
in \eqref{Herm1} or \eqref{Herm2}, respectively.
 Then, $\{  \W{ Z^\l_N}  \}_{N\in\NB}$ is a Cauchy sequence in $L^p\big(\Omega; C([0,T];W^{-\eps, \infty}(\T^4)) \big)$. Moreover, denoting the limit by $\W{Z^\l}$,
 we have $\W{Z^\l } \in C\big([0,T]; W^{-\eps,\infty}(\T^4)\big)$ almost surely, with the following tail estimate for any $1 \le q <\infty$, $T\ge1$, and $\ld>0$:
\begin{align*}
P\Big( \| \W {Z^\l } \|_{L^q_T W^{-\eps, \infty}_x} > \ld\Big) 
\leq C\exp\bigg(-c \frac{\ld^{\frac{2}{\l}}}{T^{1 + \frac{2}{q \l}}}\bigg).
\end{align*}

\noi
When $q = \infty$, we  also have 
the following tail estimate:
\begin{align}
P\Big( \|\W{ Z^\l }\|_{L^\infty ([j, j+1]; W^{-\eps, \infty}_x)}> \ld\Big) 
\leq C\exp\bigg(-c \frac{\ld^{\frac{2}{\l}}}{j+1}\bigg)
\label{P0z}
\end{align}

\noi
for any $j \in \Z_{\ge 0}$ and $\ld > 0$.

%
%
%
%

\end{lemma}
%

In order to prove  Theorem \ref{THM:GWP1}, 
we need the following finer regularity property of $ \Psi$.
For a proof, see \cite[Lemma 2.4]{GKOT} .

\begin{lemma} 
\label{LEM:log}
Let $\Psi$ be as in \eqref{defPsi}
and fix $ 0 < s < 2$.
Then, given any $x \in \T^4$ and $t \in \R_+$, 
$I\Psi(t, x)$ is a mean-zero Gaussian random variable
with variance bounded by $C_0 t \log N$, 
where the constant $C_0$ 
is independent of  $x \in \T^4$ and $t \in \R_+$.

\end{lemma}



\section{Local well-posedness of SNLB}
\label{SEC:LWP}
In this section, we show pathwise local well-posedness of the Wick renormalized SNLB \eqref{SNLB3} in Theorem \ref{THM:1}. 
In Subsection~\ref{SUBSEC:str}, we show homogeneous and inhomogeneous Strichartz estimates for the linear beam operators appearing in the mild formulation \eqref{SNLB9}. We then apply these in Subsection~\ref{SUBSEC:LWP} to show Theorem~\ref{THM:1} via a contraction mapping argument.

\subsection{Strichartz estimates}
\label{SUBSEC:str}
To obtain Strichartz estimates for the beam equation, we need the following sharp Strichartz estimates for the linear Schr\"odinger propagator $e^{\pm it\Dl}$ due to Bourgain-Demeter \cite{BD15} and Killip-Vi\c{s}an \cite{KV16}.
\begin{lemma}
\label{LEM:decoupling}
	Let $3 < p\leq\infty$ and $ N\ge1$ be a dyadic integer. Then, we have
\begin{align*}
\big\| e^{\pm it\Delta} \P_{\leq N} f \|_{L^p_{t,x}([0,1]\times \T^4)} \les N^{2-\frac6p} \|f\|_{L^2_x(\T^4)}.
\end{align*}
\end{lemma}

For $1\le q,r\le \infty$, we define the index $s_{q,r}$ as follows
\begin{equation}\label{sqr}
	s_{q,r} \deff 2 - \tfrac2q - \tfrac4r.
\end{equation}
We then obtain the following estimate.
	\begin{lemma}\label{LEM:str}
		For $3< q\leq r\leq\infty$ and $s_{q,r}$ in \eqref{sqr}, we have
		$$ \|e^{\pm it\Delta}f\|_{L_t^q([0,1]; L^r_x(\T^4))}\lesssim \|f\|_{H^{s_{q,r}}(\T^4)}.
		$$ 
	\end{lemma}
	\begin{proof}
		Let $I=[0,1]$.
		We start by writing $f = \sum_{N} \P_N f$ from Littlewood-Paley decomposition. From Bernstein's inequality and Lemma~\ref{LEM:decoupling}, we have
		\begin{align*}
			\|e^{\pm it\Delta} \P_N f \|_{L^q_I L^r_x} & \les N^{\frac4q - \frac4r} \|e^{\pm it\Delta} \P_{\leq 2N} ( \P_N f) \|_{L^q_I L^q_x}  \les N^{2 - \frac2q - \frac4r} \|\P_N f \|_{L^2_x}.
		\end{align*}
		Using the Littlewood-Paley theorem, Minkowski's inequality, and the above estimate, we obtain
\begin{align*}
\| e^{\pm it\Delta}  f \|_{L^q_I L^r_x} 
&\les \bigg(\sum_{N \geq 1 \text{ dyadic}} N^{2s_{q,r}} \|\P_N f \|^2_{L^2_x} \bigg)^{1/2} \sim \|f\|_{H^{s_{q,r}}}.
\end{align*}

\noi
as desired.
\end{proof}

From the definition of the linear beam operator $S(t)$ in \eqref{defS}, the fact that $e^{\pm i t \Dl}$ are isometries in $H^s(\T^4)$ for any $s\in\R$, and Lemma~\ref{LEM:str}, we obtain the following homogeneous Strichartz estimates for the linear beam operator.
\begin{lemma}
\label{LEM:str_ho}
Let $S(t)$ be the linear operator in \eqref{defS}, $0<T\le 1$, $3 < q \leq r \leq \infty$, and $s\geq s_{q,r}$. Then, we have
\begin{align*}
\| S(t)(u_0,u_1)\|_{L_T^\infty H_x^s} + \| S(t)(u_0,u_1)\|_{L_T^q L_x^r}&\les \|(u_0,u_1)\|_{\H^s}.
\end{align*}
\end{lemma}


We now establish the following inhomogeneous Strichartz estimate, using a $TT^*$-argument.
\begin{lemma}
\label{LEM:str_inho}
For $3< q\leq r<\infty$ and $s_{q,r}$ as in \eqref{sqr}, we have
		\begin{align*}
			\bigg\|\int_0^t\frac{\sin(t-t') \Dl}{\Dl}F(t')dt' \bigg\|_{L_t^q([0,1];L^r_x(\T^4))} &\les \|F\|_{L_t^1H_x^{s_{q,r}-2}([0,1]\times\T^4)}.
		\end{align*}
\end{lemma}
	
\begin{proof}
Note that the zero-th frequency of $F$ can be estimated easily, and so we can assume that $F$ has mean zero below.
Let $I=[0,1]$.
First note that 
\begin{align*}
\bigg\| \int_0^t &\frac{\sin((t-t') \Dl )}{\Dl} F(t') \, dt' \bigg\|_{L^q_I L^r_x} \\ 
&\les \bigg\| \int_0^t \frac{e^{i(t-t')\Delta}}{\Delta} F(t') \, dt' \bigg\|_{L^q_I L^r_x} +  \bigg\| \int_0^t \frac{e^{-i(t-t')\Delta}}{\Delta} F(t') \, dt' \bigg\|_{L^q_I L^r_x}.
\end{align*}

\noi
Thus, we focus on estimating the first term, as the estimate for the second term follows from an analogous strategy. The operator $T$ defined by $Tu_0 = e^{it\Delta}u_0$ is a bounded operator $T: H^{s_{q,r}} \to L^q_I L^r_x$ from Lemma~\ref{LEM:str}. 
Note that we have
\begin{align*}
		\jb{Tu_0,G}_{t,x}  = \int_0^1 \int_{\T^4} e^{it\Delta} u_0(x) \wb{G(t, x)} \, dx \, dt 
		& = \sum_{n \in \Z^4} \ft{u}_0(n) \wb{ \int_0^1 e^{-it|n|^2} \ft{G}(t,n) \, dt}  = \jb{u_0, T^* G}_{x},
\end{align*}

\noi
where the dual operator $T^*$ is given by
	$$T^* G = \int_0^1 e^{-it \Delta} G(t,\cdot) \, dt,$$
	which in turn is bounded from $L^{q'}_I L^{r'}_x$ to $H^{-s_{q,r}}_x$. From the trivial boundness of $T: H^s_x \to L^\infty_I H^s_x$ for any $s\in\R$, we conclude that $T^*: L^1_I H^{s_{q,r}}_x \to H^{s_{q,r}}_x$ is also bounded. Consequently, we have that $TT^*: L^1_I H^{s_{q,r}}_x \to L^q_t L^r_x$ and
	$$TT^* G = \int_0^1 e^{i(t-t')\Delta} G(t', \cdot) \, dt'.$$
	From the Christ-Kiselev lemma (Lemma \ref{LEM:CK}), we get that 
	$$\bigg\| \int_0^t e^{i(t-t')\Delta} G(t', \cdot) \, dt' \bigg\|_{L^q_I L^r_x} \les \|G\|_{L^1_I H^{s_{q,r}}_x},$$
	and by choosing $G = \frac{1}{\Delta} F$, we obtain the intended result.
	\end{proof}

\subsection{Proof of Theorem \ref{THM:1}}
\label{SUBSEC:LWP}

In this subsection, we prove Theorem \ref{THM:1} by constructing a solution $u=\Psi+v$ where $\Psi$ denotes the stochastic convolution solving \eqref{SNLB1b} and the remainder $v$ solves \eqref{eq:v_wick}. In particular, we consider the following mild formulation for $v$:
\begin{align}
v(t) = S(t) (u_0, u_1) \mp \int_0^t \frac{\sin ((t - t') \Dl)}{\Dl} \sum_{\l = 0}^k \binom{k}{\l} \Xi_\l \, v^{k - \l} (t') dt'
\label{Duhv}
\end{align}

\noi
for given initial data $(u_0, u_1)$ and a source $(\Xi_0, \Xi_1, \dots, \Xi_k)$ with the understanding that $\Xi_0 \equiv 1$, where $S(t)$ is the linear propagator as defined in \eqref{defS}.
Given $s, \eps \in \R$, we define the space $\mathcal{X}^{s,\eps} (\T^4)= \H^s (\T^4) \times \big( C([0, 1]; W^{-\eps, \infty} (\T^4)) \big)^{\otimes k}$
with the following norm for $\pmb{\Xi} = (u_0, u_1, \Xi_1, \dots, \Xi_k) \in \mathcal{X}^{s, \eps} (\T^4)$:
\begin{align*}
\| \pmb{\Xi} \|_{\mathcal{X}^{s, \eps}} = \| (u_0, u_1) \|_{\H^s} + \sum_{j = 1}^k \| \Xi_j \|_{C([0, 1]; W^{-\eps, \infty})}.
\end{align*}
Moreover, we introduce our solution space $X^{s,q,r}(T)$ for $s\in\R$ and $1 \le q,r\le \infty$:
\begin{align*}
X^{s, q, r} (T) \deff C([0, T]; H^s (\T^4)) \cap L^q ([0, T]; L^r (\T^4)).
\end{align*}

The local well-posedness in Theorem~\ref{THM:1} follows from local well-posedness of \eqref{Duhv} and Lemma \ref{LEM:sto_cov}, which states that the random enhanced data set $(u_0, u_1, \Psi, \Ws{\Psi^2}, \dots, \Ws{\Psi^k})$ almost surely belongs to $\mathcal{X}^{s, \eps}(\T^4)$ for any $\eps > 0$. We then show the following deterministic result for \eqref{Duhv}.
\begin{proposition}
\label{PROP:LWPv}
Given an integer $k \geq 2$, let $s_{\text{crit}}$ be as defined in \eqref{eq_scrit}. Then, the mild formulation \eqref{Duhv} is locally well-posed in $\mathcal{X}^{s, \eps}(\T^4)$ for 
\[ \textup{(i) } k \geq 4: 
\ s \geq  s_\textup{crit}
\qquad \text{or} \qquad 
 \textup{(ii) }  k = 2, 3: 
\ s >  s_\textup{crit},
\]
and $\eps > 0$ sufficiently small. More precisely, given an enhanced data set
\begin{align*}
\pmb{\Xi} = (u_0, u_1, \Xi_1, \dots, \Xi_k) \in \mathcal{X}^{s, \eps}(\T^4),
\end{align*}

\noi
there exist $T = T (\pmb{\Xi}) \in (0, 1]$ and a unique solution $v$ to the mild formulation \eqref{Duhv} in the class $X^{s', q, r} (T)$ for $s' = \min (s, 2 - \eps)$ and for some appropriate $1 \le q,r\le \infty$.
\end{proposition}

\begin{proof}
We define the map $\G$ by
\begin{align}
\begin{split}
\G [v] (t) &\deff S(t) (u_0, u_1) \mp \int_0^t \frac{\sin ((t - t') \Dl)}{\Dl} \sum_{\l = 0}^k \binom{k}{\l} (\Xi_\l  v^{k - \l} )(t') dt' \\
&\deff S(t) (u_0, u_1) \mp \sum_{\l = 0}^k \binom{k}{\l} \mathcal{I} \big( \Xi_\l v^{k - \l} \big) (t),
\end{split}
\label{defG}
\end{align}

\noi
and consider the following three cases.

\smallskip \noi
\underline{\textbf{Case 1:} $k \geq 4$ and $s > s_{\text{crit}}$.}

\noi Let $\eps>0$ sufficiently small and $(q, r) = (k - 1 + \ta, 2k - 2)$ for $\ta > 0$  such that $s'\ge s_{q,r}>s_{\text{crit}}$ for $s_{q,r}$ in \eqref{sqr}.
%
For $\l = 0$, by Lemma \ref{LEM:str_inho}, Sobolev's inequality and H\"older's inequality, we obtain
\begin{equation}
\begin{aligned}
\big\| \mathcal I \big( v^k \big) \big\|_{X^{s', q,r} (T)} 
&\les \|v^k\|_{L^1_T  H^{s'-2}_x}  \les  \|v^k\|_{L^1_T  L^{\frac{4}{4-s'}}_x} \les  \|v\|_{L^\infty_T L^{\frac{4}{2-s'}}_x}  \|v^{k-1}\|_{L^1_T L^{2}_x} \\
& \les   \|v\|_{L^\infty_T H^{s'}_x}  \|v\|^{k-1}_{L^{k-1}_T L^{2k-2}_x} \les T^{\eta}  \|v\|^k_{X^{s', q,r}(T)}
\end{aligned}
\label{eq_subc1-1}
\end{equation}

\noi
for some $\eta > 0$. For $1 \leq \l \leq k - 1$, proceeding as before, with Lemma~\ref{LEM:GKOlbnz}~(ii) and Lemma \ref{LEM:GKOlbnz} (i) repetitively, we obtain
\begin{align}
\begin{split}
\big\| \mathcal I \big( \Xi_\l v^{k - \l} \big) \big\|_{X^{s', q,r} (T)} 
&
= \big\|  \jb{\nb}^{-\eps} (\Xi_\l v^{k-\l}) \big\|_{L^1_T H^{s'-2+\eps}_x} \\
&\les
\big\|  \jb{\nb}^{-\eps} (\Xi_\l v^{k-\l}) \big\|_{L^1_T L^{\frac{4}{4-s'-\eps }}_x}\\
 &\les \|  \jb{\nb}^{-\eps} \Xi_\l  \|_{L^\infty_T L^{\frac{4}{\eps}}_x }    \big\|  \jb{\nb}^\eps v^{k-\l}  \big\|_{L^1_T L^{\frac{4}{4-s'-\eps }}_x}\\
  &\les \|  \Xi_\l  \|_{L^\infty_T W^{-\eps, \infty}_x }    \|  \jb{\nb}^\eps v    \|_{L^\infty_T  L^{\frac{4}{2-s'+\eps }}_x}
   \|  v    \|^{k-\l-1}_{L^{k-\l-1}_T  L^{\frac{2(k-\l-1)}{1-\eps }}_x}  \\
&\les \| \Xi_\l \|_{L^\infty_T W^{-\eps, \infty}_x }     \|    v    \|_{L^\infty_T  H^{s'}_x }
   T^\eta    \|  v    \|^{k-\l-1}_{L^{q}_T  L^{ r}_x} \\
 &\les    T^\eta \| \Xi_\l \|_{L^\infty_T W^{-\eps, \infty}_x }       \|  v    \|^{k-\l}_{X^{s', q,r}(T)}
\end{split}
\label{eq_subc1-2}
\end{align}

\noi
for some $\eta > 0$ and $\eps>0$ sufficiently small. 
Lastly, for $\l = k$, by Lemma \ref{LEM:str_inho}, since $s'<2$, we have
\begin{align}
\begin{split}
\big\| \mathcal I ( \Xi_k ) \big\|_{X^{s', q,r} (T)} 
& \les \| \Xi_k \|_{L^1_T H^{s'-2}_x} \les T \| \Xi_k \|_{L_T^\infty W_x^{-\eps, \infty}}.
\end{split}
\label{eq_subc1-3}
\end{align}

\noi
By Lemma \ref{LEM:str_ho}, \eqref{defG}, \eqref{eq_subc1-1}, \eqref{eq_subc1-2}, and \eqref{eq_subc1-3}, we have
\begin{align*}
\| \G[v] \|_{X^{s', q,r} (T)} 
&\les \| \pmb{\Xi} \|_{\mathcal{X}^{s , \eps}} + T^\eta \big[ \| \pmb{\Xi} \|_{\mathcal{X}^{s , \eps}}^k + \| v \|_{X^{s', q,r} (T)}^k \big].
\end{align*}

A straightforward modification of the above steps yields the following difference estimate:
\begin{align*}
\| \G[v_1] - \G[v_2] \|_{X^{s', q,r} (T)} &\les T^\eta \big[ \| \pmb{\Xi} \|_{\mathcal{X}^{s , \eps}}^k +  \| v_1 - v_2 \|_{X^{s', q,r} (T)}  \big( \| v_1 \|_{X^{s', q,r} (T)}^{k - 1} + \| v_2 \|_{X^{s', q,r} (T)}^{k - 1} \big) \big].
\end{align*}

\noi
Then, by $T = T (\| \pmb{\Xi} \|_{\mathcal{X}^{s, \eps}})>0$  sufficiently small, the local well-posedness of \eqref{Duhv} on $[0, T]$ follows from a contraction mapping argument.

\smallskip \noi
\underline{\textbf{Case 2:} $k = 2, 3$ and $s > s_{\text{crit}} = 0$.}

\noi
In this case, we take $(q,r) = (3+\ta, 3+\ta)$ for $0<\ta\le\frac{3s'}{2-s'}$ which guarantees that $s' \ge s_{q,r}$.
For $\l = 0$, proceeding as in \eqref{eq_subc1-1}, we have
\begin{equation}
\big\| \mathcal I \big( v^k \big) \big\|_{X^{s', q,r} (T)} 
 \les \| v^k\|_{L^1_T H^{s'-2}_x} \les \| v^k\|_{L^1_T L^{\frac{4}{4-s'}}_x} \les  \|v\|_{L^k_T L_x^{\frac{4k}{4-s'}}} \les T^\eta \|v\|_{X^{s',q,r}(T)},
\label{eq_subc2-1}
\end{equation}

\noi
for some $\eta > 0$, since $\frac{4k}{4-s'} \le \frac{6}{2-s'}$ for $k=2,3$. For $1 \leq \l \leq k - 1$, noticing that $\frac{2 (k - \l - 1)}{1 - \eps} < 3 + \ta$ for $k \leq 3$ and $\eps > 0$ sufficiently small, we proceed as in \eqref{eq_subc1-2} to obtain
\begin{align}
\begin{split}
\big\| \mathcal I \big( \Xi_\l v^{k - \l} \big) \big\|_{X^{s', q,r} (T)} =  T^\eta \| \Xi_\l \|_{L_T^\infty W_x^{-\eps, \infty}} \| v \|_{X^{s', q,r} (T)}^{k - \l}
\end{split}
\label{eq_subc2-2}
\end{align}

\noi
for some $\eta > 0$. 
%
By Lemma \ref{LEM:str_ho}, \eqref{defG}, \eqref{eq_subc2-1}, \eqref{eq_subc2-2}, and \eqref{eq_subc1-3}, we have
\begin{align*}
\| \G[v] \|_{X^{s', q,r} (T)} \les \| \pmb{\Xi} \|_{\mathcal{X}^{s , \eps}} + T^\eta \| \pmb{\Xi} \|_{\mathcal{X}^{s , \eps}}^k + T^\eta \| v \|_{X^{s', q,r} (T)}^k.
\end{align*}

\noi
Similar steps yield a difference estimate and we conclude the argument as in Case 1.

\smallskip \noi
\underline{\textbf{Case 3:} $k\ge4$ and $s = s_{\text{crit}}$.}

\noi
In this case, we take $
(q, r) = (k , \frac{2k(k-1)}{k+1})$ 
so that $s'=s=s_{\text{crit}} = s_{q,r}$ and $q,r>3$. By proceeding as in Case 1, the estimates \eqref{eq_subc1-2} and \eqref{eq_subc1-3} hold, but we can only show \eqref{eq_subc1-1} without the gain of $T^\eta$ on the right-hand side. Thus, we have
\begin{align}
\| \G [v] \|_{L^\infty_T H^{s'}_x} &\leq C \| (u_0, u_1) \|_{\H^{s}} + C T^\eta \| \pmb{\Xi} \|_{\mathcal{X}^{s, \eps}} \| v \|_{L^\infty_T H_x^{s'}} \sum_{\l = 0}^{k - 2} \| v \|_{L_T^{q} L_x^{r}}^\l \nonumber\\
&\quad + C T \|\pmb\Xi \|_{\mathcal{X}^{s,\eps}} + C \| v \|_{L^q_T L^r_x}^{k} ,
\label{eq_subc3-1}\\
\| \G [v] \|_{L_T^{q} L_x^{r}} &\leq C \| S(t) (u_0, u_1) \|_{L_T^q L_x^{r}} + C T^\eta \| \pmb{\Xi} \|_{\mathcal{X}^{s, \eps}} \| v \|_{L^\infty_T H_x^{s'}} \sum_{\l = 0}^{k - 2} \| v \|_{L_T^{q} L_x^{r}}^\l \nonumber\\
&\quad + C T \|\pmb\Xi \|_{\mathcal{X}^{s,\eps}} + C \| v \|_{L^q_T L^r_x}^{k} ,
\label{eq_subc3-2}
\end{align}

\noi
for some $C > 0$ and $\eta > 0$.
We now define the set $B_{a, b, T}$ as
\begin{align*}
B_{a, b, T} \deff \big\{ v \in X^{s', q,r} (T): \| v \|_{L^\infty_T H^{s'}_x} \leq a \text{ and } \| v \|_{L_T^{q} L_x^{r}} \leq b \big\}.
\end{align*}

\noi
Suppose that $\| (u_0, u_1) \|_{\H^s} \leq A$ for some $A > 0$. We let $a = 4CA$ and $0 < b \leq 1$ small enough such that
\begin{align}
C b^{k} \le \min(\tfrac a 4, \tfrac b 4).
\label{bdd1}
\end{align}

\noi
By dominated convergence theorem, we can let $T = T(u_0, u_1) > 0$ be small enough so that
\begin{align}
\| S(t) (u_0, u_1) \|_{L_T^{q} L_x^{r}} \leq \tfrac{b}{4C}.
\label{bdd2}
\end{align}

\noi
Choosing $T$ smaller, if necessary, we also assume that
\begin{align}
CT^\eta (k - 1) \| \pmb{\Xi} \|_{\mathcal{X}^{s, \eps}} \leq \min (\tfrac{1}{4a}, \tfrac{1}{4b}) \qquad \text{and} \qquad CT \| \pmb{\Xi} \|_{\mathcal{X}^{s, \eps}} \leq \min (\tfrac a 4, \tfrac b4)
\label{bdd3}
\end{align}

\noi
Combining \eqref{eq_subc3-1}, \eqref{eq_subc3-2}, \eqref{bdd1}, \eqref{bdd2}, and \eqref{bdd3}, we know that for $v \in B_{a, b, T}$, we have
\begin{align*}
\| \G [v] \|_{L^\infty_T H^{s'}} \leq a \qquad \text{and} \qquad 
\| \G [v] \|_{L_T^{q} L_x^{r}} \le b, 
\end{align*}

\noi
so that $\G$ maps $B_{a,b,T}$ to $B_{a,b,T}$. By further shrinking $b$ and $T$ if necessary, we can use similar steps to obtain
\begin{align*}
\| \G[v_1] - \G[v_2] \|_{X^{s', q,r} (T)} \leq \tfrac 12 \| v_1 - v_2 \|_{X^{s', q, r} (T)},
\end{align*}

\noi
so that $\G$ is a contraction map on $B_{a,b,T}$. We can then conclude the proof of local well-posedness of \eqref{Duhv}. \qedhere

\end{proof}

\begin{remark}\rm \label{REM:LWP-sharp}
	Note that in Case 3 above, to extend the argument to cover the critical regularity $s=0$ for $k=2,3$ (even without the noise terms), we would need to find suitable $q,r$ such that $s_{q,r}=0$ with $s_{q,r}$ in \eqref{sqr}. However, we can easily see that this requires that $q> 3$ which implies that $r< 3$ and vice-versa, thus the Strichartz estimates in Lemmas~\ref{LEM:str_ho}-\ref{LEM:str_inho} do not apply. Moreover, since these are derived from the sharp Strichartz estimates in Lemma~\ref{LEM:decoupling} which are known to fail at the endpoint $p=3$ \cite{BO93}, the argument above is insufficient to reach critical regularity for quadratic and cubic nonlinearities.
	
\end{remark}

\section{Pathwise global well-posedness of the cubic SNLB}
\label{SEC:GWP1}
In this section, we show pathwise global well-posedness of the  Wick-ordered cubic SNLB \eqref{SNLB6} via the hybrid argument in \cite{GKOT}. We restrict our attention to $0<s<2$, since the result for $s\ge 2$ follows from the same argument.
In Subsection~\ref{SUBSEC:3.1}, we first show some preliminary estimates involving the $I$-operator, and establish commutator estimates to control \eqref{E1}. We then prove Theorem~\ref{THM:GWP1} in Subsection~\ref{SUBSEC:3.2}.

\subsection{Commutator estimates and other preliminaries}
\label{SUBSEC:3.1}
We recall the definition of the $I$-operator with Fourier multiplier $m_N$ in \eqref{I0a}. In the following, we fix $N \in \NB$.
	From the definition of the $I$-operator and the Littlewood-Paley theorem, we have that
	\begin{align}
		\|f\|_{H^s}\les \|I f\|_{ H^2} &\les N^{2-s} \|f\|_{H^s}
		\label{I1},\\
		\|I f\|_{W^{s_0 + s_1, p}} &\les  N^{s_1}\|f\|_{W^{s_0, p}},
		\label{I2}
	\end{align}
	
	\noi
	for any $s_0 \in \R$, $0 \leq s_1 \leq 2- s$,  and $1 < p < \infty$.
For simplicity, we will use the notations
\begin{align}
f_{\les N} \deff \pi_{\frac{N}{3}} f
\qquad \text{and}\qquad  f_{\ges N} \deff \pi_{\frac{N}{3}}^\perp f \deff f - f_{\les N},
\label{Q0}
\end{align}
where $\pi_N$ denotes the projection onto frequencies $\{|n| \le N\}$.

We first go over some basic commutator estimates in the following lemmas.

\begin{lemma}\label{LEM:C1}
Let $\frac 43 \le s < 2$. Then, for $k=1,2,3$, we have  
\begin{align*}
\|(If)^k - I(f^k)\|_{L^2} \les N^{-2+k(2-s)} \|If\|^k_{H^2}.
\end{align*}

\end{lemma}

\begin{proof}
By the definition of the $I$-operator and \eqref{Q0},  we have 
 $I(f_{\les N}^k) = f_{\les N}^k$
for $ k = 1, 2, 3$.
Thus, we obtain
\begin{equation}
\begin{aligned}
(If)^k - I(f^k) 
&= \big(I(f_{\les N} + f_{\ges N})\big)^k - I\big((f_{\les N} + f_{\ges N})^k\big) \\
&= \big(f_{\les N} + I(f_{\ges N})\big)^k - I\big((f_{\les N} + f_{\ges N})^k\big)\\
&= f_{\les N}^k - I\big(f_{\les N}^k\big) + 
\sum_{j=0}^{k-1} {k \choose j}
\Big(f_{\les N}^j(If_{\ges N})^{k-j} - I\big(f_{\les N}^j f_{\ges N}^{k-j}\big)
\Big) \\
&= \sum_{j=0}^{k-1} {k \choose j}
\Big(f_{\les N}^j(If_{\ges N})^{k-j} - I\big(f_{\les N}^j f_{\ges N}^{k-j}\big)
\Big).
\end{aligned}
\label{Q2}
\end{equation}

We first consider the case when $1 \leq j \leq k - 1$. We let $1<q<\infty$ sufficiently large and $\dl > 0$ small such that $\frac 12 = \frac{j}{q} + \frac{1}{2 + \dl}$.
 Then, by H\"older's and Sobolev's inequalities, we have
\begin{align}
\begin{split}
\|f_{\les N}^j (If_{\ges N})^{k-j}\|_{L^2}
& \le 
\|f_{\les N}\|_{L^{q}}^{j}
\|If_{\ges N}\|_{L^{(2+\dl) (k - j)}}^{k-j} \\
& \les 
 \|f_{\les N} \|_{H^{2}}^{j} 
\|If_{\ges N}\|_{H^{2 - \frac{4}{(2+\dl) (k - j)}}}^{k-j}\\
& \les N^{-\frac{4}{2+\dl}}  \|If\|_{H^2}^k.
\end{split}
\label{Q3}
\end{align}

\noi
Similarly,
using the boundedness of the multiplier $m_N$
and \eqref{I1}, we have  
\begin{align}
\begin{split}
\big\|I\big(f_{\les N}^j f_{\ges N}^{k-j}\big)\big\|_{L^2} 
& \les \|f_{\les N}^j f_{\ges N}^{k-j}\|_{L^2} \\
& \le \|f_{\les N}\|_{L^q}^j \|f_{\ges N}\|_{L^{(2+\dl)(k-j)}}^{k - j} \\
& \les 
\|I f_{\les N}\|_{H^2}^j
\| f_{\ges N} \|_{H^{2 - \frac{4}{(2+\dl) (k - j)}}}^{k-j}\\
& \les 
N^{ (k-j)(2-s)-\frac{4}{2+\dl}} 
\|I f_{\les N}\|_{H^2}^j
\| f_{\ges N} \|_{H^s}^{k-j}\\
& \les N^{-2 + k(2-s)} \|If \|_{H^2}^k.
\end{split}
\label{Q4}
\end{align}

\noi
Here, we have used the fact that $2 - \frac{4}{(2 + \dl) (k - j)} \leq s$, which is guaranteed 
by $ \frac 43 \le s < 2$. When $j = 0$, similar estimates to \eqref{Q3} and \eqref{Q4} hold with $q = \infty$ and $\dl = 0$.
Therefore, the desired estimate
follows from \eqref{Q2}, \eqref{Q3}, and~\eqref{Q4}.
\end{proof}

\begin{lemma} \label{LEM:C2}
 Let $0 < s < 2$ and $0< \g < 1$. 
 Given $\dl = \dl(s) > 0$ sufficiently small, there exist small $\g_0 = \g_0(\dl) > 0$
 and large $p = p(\dl) \gg1$ such that  
\begin{align*}
\norm{(If)(Ig) - I(fg)}_{L^2} 
\les N^{ - \frac{1-\g}{2} + \dl } \|f\|_{H^{2-\g}}\|g\|_{W^{-\g_0, p}}
\end{align*}

\noi
for any sufficiently large $N \gg1$.
\end{lemma}

\begin{proof}
By writing $f = f_{\les N^{\frac 12}} + f_{\ges N^{\frac12}}$ 
 and $g= g_{\les N} + g_{\ges N}$, we have 
 \begin{align*}
 \begin{split}
 (If)(Ig) - I(fg) 
 &= \Big\{ (If_{\les N^{\frac12}}) (Ig_{\les N}) 
 - I(f_{\les N^{\frac12}}g_{\les N})\Big\}\\
& \hphantom{X}
+ \Big\{ (If_{\les N^{\frac12}})(Ig_{\ges N}) - 
I(f_{\les N^{\frac12}}g_{\ges N})\Big\}\\
& \hphantom{X}
+ (If_{\ges N^{\frac12}})(Ig) - I(f_{\ges N^{\frac12}}g)\\
& =: B_1 + B_2 + B_3 + B_4.
\end{split}
 \end{align*}

\noi Since the Fourier support of $f_{\les N^\frac12 } g_{\les N}$ is contained in $\{|n| \le \frac23 N\}$, then $B_1 \equiv 0$.

%

For $B_2$, note that for $(n_1,n_2) \in \Ld_n \deff \{(x,y)\in \Z^4 \times \Z^4: \ n= x +y , \, |x| \le \frac{N^\frac12}{3}, \, |y| > \frac N 3 \}$, by considering the sub-regions $|n_2|\ge 3N$ and $|n_2| < 3N$, from the mean value theorem and the definition in \eqref{I0a}, we get
\begin{align*}
|m_N(n_1+n_2)-m_N(n_2)| \les  N^{2-s}|n_2|^{-3+s}
|n_1|.
\end{align*}

\noi
From the above, the fact that $m_N(n_1) \equiv 1$
on $\Ld_n$, and Cauchy-Schwarz inequality,
we have
 \begin{align*}
\begin{split}
\| B_2\|_{L^2} 
& = \bigg\|\sum_{(n_1, n_2) \in \Ld_n }
\big(m(n_2)-m(n_1+n_2)\big)\ft f (n_1)\ft g(n_2)\bigg\|_{\l^2_n}\\
& \les N^{2-s}
\Bigg\| \sum_{(n_1, n_2) \in \Ld_n} \frac{1}{\jb{n_1}^{1-\g} \jb{n_2}^{3-s-\dl}} \jb{n_1}^{2-\g} |\ft{f}(n_1)| \frac{|\ft{g}(n_2)|}{\jb{n_2}^\dl}
\Bigg\|_{\l^2_n}\\
& \les N^{-\frac12 + \dl}
\|f\|_{H^{2-\g}}\|  g\|_{ H^{-\dl} }.
\end{split}
\end{align*}

As for $B_3$, by H\"older's inequality, Sobolev's embedding, and \eqref{I2}, 
we have
\begin{align*}
\begin{split}
\|B_3\|_{L^2} 
& \le \|If_{\ges N^{\frac12}}\|_{L^2} \|Ig\|_{L^\infty}  \les N^{-1+\frac{\g}{2}} \|f \|_{H^{2-\g}}\|Ig\|_{W^{5\delta,\delta^{-1}}} \\
& \les   N^{-1+\frac{\g}{2}+ 6 \delta} \|f\|_{H^{2-\g}}\|g\|_{W^{-\delta,\delta^{-1}}} 
\end{split}
\end{align*}

\noi
for $\dl = \dl(s) > 0$ sufficiently small.

Lastly, by
\eqref{I2}
and Lemma \ref{LEM:GKOlbnz}\,(ii), we have
\begin{align*}
\begin{split}
\|B_4\|_{L^2} 
& \les  N^{4\dl} \| f_{\ges N^\frac12} g \|_{H^{-4\dl}} \les N^{4\dl}  \| f_{\ges N^\frac{1}{2}}\|_{H^{4\dl}} \|g\|_{W^{-4\dl , \dl^{-1}}}\\
& \les N^{-1+\frac {\g}2  + 6\dl}  \| f\|_{H^{2-\g}} \|g\|_{W^{-4\dl , \dl^{-1}}}
\end{split}
\end{align*}

\noi
for $\dl = \dl(s) > 0$ sufficiently small. \qedhere

 \end{proof}

We now show the following commutator estimate using Lemma \ref{LEM:C1} and Lemma \ref{LEM:C2}.

 
\begin{lemma}\label{LEM:C3}
Let $\frac 32 < s < 2$
and  $k =  1,  2$. 
 Given $\dl = \dl(s) > 0$ sufficiently small, there exist small $\g_0 = \g_0(\dl) > 0$
 and $p = p(\dl) \gg1$ such that  
\begin{align*}
\| I(f^k g) - (If)^k I g \|_{L^2} 
\les
N^{-\frac{1-k(2-s)}{2} + \dl } \|If\|_{H^2}^k \|g\|_{W^{-\g_0 ,p}}
\end{align*}

\noi
for sufficiently large $N \gg1$.

\end{lemma}

\begin{proof}
Using triangle inequality, we have 
\begin{align*}
\| I ( f^kg ) - (If)^k Ig\|_{L^2} 
&\le \| I(f^kg) -I(f^k)Ig\|_{L^2}
+ \big\|\big(I(f^k) - (If)^k\big)Ig\big\|_{L^2} =: D_1 + D_2.
\end{align*}

For $D_1$, by Sobolev's inequality and Lemma \ref{LEM:GKOlbnz}(i),
we have
\begin{align} 
\|f^{k}\|_{H^{2-k(2-s)}} 
\les \|f^k\|_{W^{s,\frac{4}{2+(k-1)(2-s)}}}
\les \|f\|_{H^{s}} \|f\|_{L^{\frac4{2-s}}}^{k-1} \les\|f\|_{H^{s}}^k.
\label{Q33a}
\end{align}

\noi
Thus, by  Lemma \ref{LEM:C2} with $\g = k(2-s)$, 
\eqref{Q33a}, 
and \eqref{I1}, there exists $\dl > 0$ sufficiently small such that
\begin{align*}
\begin{split}
 \|D_1\|_{L^2} 
 & \les N^{-\frac{1-k(2-s)}2 + \dl} 
\|f\|^k_{H^{s}} \|g\|_{W^{-\g_0, p }}  \les N^{-\frac{1-k(2-s)}2 + \dl} 
\|If\|^k_{H^{2}} \|g\|_{W^{-\g_0, p }}
\end{split}
\end{align*}

\noi
for some small $\g_0 = \g_0 (\dl)>0$ and large $p = p (\dl) \gg1 $.

Lastly, by H\"older's inequality, Lemma \ref{LEM:C1}, Sobolev embedding, and
\eqref{I2}, we have
\begin{align*}
\begin{split}
\|D_2\|_{L^2} & 
\le \|I(f^k) - (If)^k \|_{L^2} 
\|Ig\|_{L^\infty} \\
& \les N^{-2+k(2-s)} \|If\|_{H^2}^k \|Ig\|_{W^{5\delta,\delta^{-1}}} \\
& \les N^{-2+k(2-s) + 6\delta} \|If\|_{H^2}^k 
\|g\|_{W^{-\delta,\delta^{-1}}}
\end{split}
\end{align*}

\noi
given that $\dl = \dl(s) > 0$ is sufficiently small. \qedhere

\end{proof}

We conclude this subsection by 
showing the following
estimates, which will be useful in estimating the second and third terms in \eqref{E1}.

\begin{lemma}\label{LEM:C4}
\textup{(i)}
 Let $0 < s < 2$ and $k = 0,1$. Then, for any $0 \le \ld  \le 2-s$, 
 we have 
 \begin{align*}
\bigg|\int_{\T^4} (\dt I v(t) )(I v(t))^k I w (t) \, dx \bigg|
\les N^\ld \big(1 + \big[E(I \vec v )(t)\big]^{\frac34} \big) \| w(t) \|_{ W_x^{-\ld, 4} }
\end{align*}

\noi
for any $t \geq 0$, 
where $E$ is the energy defined  in \eqref{Hamil}.

\smallskip

\noi
\textup{(ii)}
We have
\begin{align*}
\begin{split}
\bigg|\int_{t_1}^{t_2} \int_{\T^4} (\dt  & Iv)   (Iv)^2 I w \, dx dt \bigg|  \les 
\|I w\|_{L^{\eta^{-1}}_{[t_1, t_2],x}} \int_{t_1}^{t_2} \bigg(E^{\frac{1 + \eta}{1 - 2 \eta}}(I\vec v)(t) + \frac{\eta}{(t-t_1)^\frac12}\bigg) dt
 , 
\end{split}
\end{align*}

\noi
uniformly 
in $0 < \eta < \frac 18$
and $t_2 \ge t_1 \ge 0$.

\end{lemma}

\begin{proof}(i) 
By H\"older's inequality, \eqref{Hamil}, and \eqref{I2},  we have
\begin{align*}
\bigg|\int_{\T^4} (\dt  Iv(t)) (Iv(t))^k  Iw(t) \, dx\bigg| 
& \les\|\dt Iv(t)\|_{L^2} \|Iv (t)\|_{L^4}^k \|I w(t)\|_{L^4} \\
& \les  N^\ld  \big[ E(I\vec v) (t) \big]^{\frac12 + \frac k4}
\|w(t)\|_{W^{-\ld , 4}}.
 \end{align*}

\smallskip
\noi
(ii) From H\"older's inequality, Sobolev inequality, and \eqref{E1}, we have
\begin{align}
	\bigg|\int_{t_1}^{t_2} \int_{\T^4} (\dt Iv) (Iv)^2 I w \, dx dt\bigg| 
	&\le \int_{t_1}^{t_2}
	\|\dt Iv\|_{L^2_x} 
	\|Iv\|_{L^4_x} \|Iv\|_{L^{\frac 4 {1-4\eta}}_x} \|Iw \|_{L_x^{\eta^{-1}}} dt \nonumber\\
	& \leq \int_{t_1}^{t_2} \big[ E(I \vec v )(t)\big]^{\frac34} \|I v\|_{L_x^{\frac{4}{1-4\eta}}} \| I w\|_{L_x^{\eta^{-1}}} \, dt \label{dtvv2w}
\end{align}
	for $\eta>0$.
	By Sobolev inequality, interpolation, and \eqref{Hamil}, we get
	\begin{equation}
		\begin{aligned}
			\|I v\|_{L^{\frac 4 {1-4\eta}}}
			&\les \| I v \|_{W^{8\eta, \frac{4}{1 + 4\eta}}} \les \|Iv \|_{H^2}^{4\eta} \|I v\|_{L^4}^{1-4\eta } \le \big[ E(I\vec v) \big]^{\frac{1+4\eta}{4} },
		\end{aligned}
	\label{dtvv2w-1}
	\end{equation}
	uniformly in $0<\eta<\frac18$. From \eqref{dtvv2w}, \eqref{dtvv2w-1}, and H\"older's inequality, we get
	\begin{align}
		\bigg|\int_{t_1}^{t_2} \int_{\T^4} (\dt Iv) (Iv)^2 I w \, dx dt\bigg| & \les \bigg( \int_{t_1}^{t_2} \big[ E(I \vec v) \big]^{\frac{1+\eta}{1-\eta}} \, dt \bigg)^{1-\eta} \| I w\|_{L_{[t_1,t_2], x}^{\eta^{-1}}}. \label{dtvv2w-2}
	\end{align}
	
	To estimate the first factor in \eqref{dtvv2w-2}, let 
	\begin{align*}
		p = \frac{1-\eta}{1-2 \eta}, &&
		q  = \frac{1}{1-2\eta}, && p' = \frac {1-\eta}{\eta}, &&
		q' = \frac{1}{2\eta},
	\end{align*}
	where $p',q'$ are the H\"older conjugates of $p,q$, respectively. By H\"older's and Young's inequalities, we have
\begin{align*}
\bigg(\int_{t_1}^{t_2} f (t) dt \bigg)^{1-\eta} 
&\le \bigg(\int_{t_1}^{t_2} |f(t) |^p dt \bigg)^{\frac{1-\eta}{p}} (t_2 - t_1)^{\frac{1-\eta}{p'}}\\
&\le \frac 1{q}\bigg(\int_{t_1}^{t_2} |f(t)|^pdt \bigg)^{ \frac{q(1-\eta)}{p}} + \frac{1}{q'} (t_2 - t_1)^{ \frac{q'(1-\eta)}{p'}} \\
&= (1-2\eta) \int_{t_1}^{t_2} |f(t) |^{\frac{1-\eta}{1-2 \eta}} \, dt + 2\eta (t_2 - t_1)^\frac12.
\end{align*}

\noi
Thus, we obtain
\begin{align*}
\bigg(\int_{t_1}^{t_2} E^{\frac{1+\eta}{1-\eta}}(I\vec v)(t)dt \bigg)^{1-\eta} 
\lesssim 
\int_{t_1}^{t_2} \bigg(E^{\frac{1+\eta}{1-2\eta}}(I\vec v)(t) + \frac{\eta}{(t-t_1)^\frac12}\bigg) dt.
\end{align*}

\noi
Combining the above estimates gives the intended estimate.
\end{proof}

\subsection{Proof of Theorem \ref{THM:GWP1}}
\label{SUBSEC:3.2}

In this subsection, we construct a solution to the Wick renormalized cubic SNLB \eqref{SNLB6} on the time interval $[0, T]$
for any given $T \gg 1$. 
The argument is based on that in \cite{GKOT}.

We first fix $\frac32<s<2$, $N\gg1$ sufficiently large, and $T>0$, and establish an estimate for the growth of  the modified energy $E(t)=E(I_N\vec v) (t)$
on the time interval $[0, T]$.
Note that by \eqref{Hamil} and H\"older's inequality, we have
\begin{align}
\| I v \|_{H^2}^2 = \| I v \|_{L^2}^2 + \| \Dl (I v) \|_{L^2}^2 \leq 2 E^{\frac 12} (t) + 2 E(t).
\label{H2E}
\end{align}

\noi
Then,  by \eqref{E1}, Cauchy-Schwarz inequality, Lemmas \ref{LEM:C1},\ref{LEM:C3}-\ref{LEM:C4}, and~\eqref{H2E}, 
we have 
\begin{align} 
\begin{split}
E(t_2)  - E(t_1) 
&\les
 \int_{t_1}^{t_2} N^{-2+3(2-s)} \big( 1 + E^2(t) \big) dt \\
&\quad +  \sum_{k=1}^2 \int_{t_1}^{t_2} N^{-\frac{1-k(2-s)}2 + \dl}
\big( 1 + E^\frac {k+1} 2(t) \big) \| \Ws{  \Psi^{3-k}(t) }  \|_{W^{-\g_0, p }_x} dt \\
&\quad + \sum_{k=2}^3 \int_{t_1}^{t_2} N^\ld \big( 1 + E^{\frac34}(t) \big) 
\|  \Ws{  \Psi^{k}(t)}  \!\|_{ W_x^{-\ld, 4}}  dt \\
&\quad + 
\Bigg\{\int_{t_1}^{t_2} \bigg(E^{\frac{1 + \eta}{1 - 2 \eta}}(t) + \frac{\eta}{(t-t_1)^\frac12}\bigg) dt\Bigg\}
\|I \Psi\|_{L^{\eta^{-1}}_{[t_1, t_2],x}}
\end{split}
\label{Y1}
\end{align}

\noi
for any $t_2 \geq t_1 \geq 0$, where $\g_0 = \g_0 (\dl) > 0$ is sufficiently small, $p=p(\dl)\gg1$ sufficiently large, $0 \leq \ld \leq 2 - s$, and $0 < \eta < \frac 18$.

Before proceeding to the iterative argument, we introduce some notations.
Given $j \in \Z_{\ge 0}$, we define $V_j = V_j(\o)$ by
\begin{align*}
V_j =  \max_{k = 1, 2} \|  \Ws{  \Psi^{3-k}}  \|_{L^\infty_{[j, j+1]}W^{-\g_0, p}_x} 
+ \max_{k = 0, 1}\|\Ws{  \Psi^{3-k}  }\|_{L^\infty_{[j, j+1]}W^{-\ld, 4}_x} 
\end{align*}

\noi
and define $V = V(\o)$ by 
\begin{align}
e^{V^\frac{1}{3}} = \sum_{j = 0}^\infty e^{-j K} e^{V_j^\frac{1}{3}}
\label{K2}
\end{align}

\noi
for some  $K > 0$ large enough.
Note that by applying 
\eqref{P0z} in 
Lemma~\ref{LEM:sto_cov}
and letting $K > 0 $ be sufficiently large, we have 
\begin{align*}
 \E\big[ 
e^{V^\frac{1}{3}} \big] = \sum_{j = 0}^\infty e^{-j K} \E\big[e^{V_j^\frac{1}{3}}\big]
\leq \sum_{j = 0}^\infty e^{-j K} e^{c (j+1)} < \infty,
\end{align*}

\noi
so that $V$ is almost surely finite. Also, for $T>0$, we define $M_T = M_T(\o)$ as follows
\begin{align}
M_T =  \max_{k = 1, 2} \| \Ws{  \Psi^{3-k}  } \|_{L^\infty_TW^{-\g_0, p}_x} 
+ \max_{k = 0, 1}\|  \W{ \Psi^{3-k}  }\|_{L^\infty_TW^{-\ld, 4}_x} .
\label{K2a}
\end{align}

\noi
From \eqref{K2} we have that 
$ V_j^\frac{1}{3} \leq V^\frac{1}{3} + j K$, and
therefore
\begin{align}
M_T =  \max_{j \leq T} V_j \les V + K^3 T^3.
\label{K3}
\end{align}

\noi
Furthermore, we define $R = R(\o)$ by 
\begin{align}
R = 1+ \sum_{N = 1}^\infty \sum_{j = 1}^\infty
e^{- j K \log N}
\int_0^j \int_{\T^4}e^{|I_N \Psi (t,x) |} dx dt.
\label{K4}
\end{align}

\noi
Then, by using Lemma \ref{LEM:log}
and taking $K > 0  $ possibly larger, we have 
\begin{align*}
\E[ R ] 
& = 1+\sum_{N = 1}^\infty \sum_{j = 1}^\infty
e^{- j K \log N} 
\int_0^j \int_{\T^4}\E \Big[e^{|I_N \Psi (t,x) |}\Big] dx dt\\
& \les  \sum_{N = 1}^\infty \sum_{j = 1}^\infty
e^{- j K \log N}
j  e^{c j\log N} < \infty.
\end{align*}

\noi
Therefore, $1\le R(\o) <\infty$ almost surely.

In the following, we fix $\o\in\O$, where $\O$ is the full probability set where for all $\o\in\O$ we have $V(\o), R(\o)<\infty$, and prove pathwise well-posedness of \eqref{SNLB6} on $\O$. We first need the following crucial result.

\begin{proposition} \label{PROP:LGWP}
Let $\frac 32 <  s < 2$,  $T \ge T_0 \gg 1$, and $N \in \NB$ with $N > 10$.
Let $V = V(\o) < \infty$ and $R = R(\o) < \infty$
be as in \eqref{K2} and \eqref{K3}.
Then, there exist $0<\al\le 2s-3$ and  $0<\be<\al$ such that 
if \begin{align}
E(t_0) \le N^\beta
\label{Y2}
\end{align}

\noi
for some $0 \le t_0 < T$, then
there exists small
$\tau = \tau (s, T, K, \o) >0 $
such that 
\begin{align*}
E(t) \le N^\alpha
\end{align*} 

\noi
for any $t $
satisfying
$t_0 \le t \le \min(T, t_0 + \tau)$. 
\end{proposition}

\begin{proof}
By replacing
$E(t)$ by $E(t) + 1$, we can assume that $E(t) \geq 1$.
Then, from \eqref{Y1} with \eqref{K2a}, we have
\begin{align} 
\begin{split}
E(t) \,  -  \,  E(t_0) 
&\les
\, \int_{t_0}^{t} N^{-2+3(2-s)} E^2(t') dt' \\
&\quad +  M_T \sum_{k=1}^2 \int_{t_0}^{t} N^{-\frac{1-k(2-s)}2 + \dl}
E^\frac {k+1} 2(t') dt' \\
&\quad + M_T  \int_{t_0}^{t} N^\ld  E^{\frac34}(t') 
  dt' \\
&\quad + 
\Bigg\{\int_{t_0}^{t} \bigg(E^{1+c\eta}(t') + \frac{\eta}{(t'-t_0)^\frac12}\bigg) dt'\Bigg\}
\|I \Psi\|_{L^{\eta^{-1}}_{[t_0, t],x}}
\end{split}
\label{K5}
\end{align}

\noi
for any $t \geq t_0$ and for $c= \frac{3}{1-2\eta}>0$.

We assume that \eqref{Y2} holds for some $0 \leq t_0 < T$. By the continuity in time of $E(t)$
and~\eqref{Y2} with $\al > \be$, there exists $t_1 > t_0$ sufficiently close to $t_0$ such that
\begin{align}
\max_{t_0 \leq \tau \leq t} E(\tau) \leq 100 N^{\al}
\label{K5a}
\end{align}

\noi
for any $t_0 \leq t \leq t_1$, where $\al > \beta$ is to be determined later. Note that at this point, $t_1$ depends on $t_0$. This issue will be dealt later.

Let $\eta = \frac{1}{n}$ for some $n  \in \NB$. 
We note from 
\eqref{K4} and $n! \le n^n$ that
\begin{align*}
\| I_N \Psi\|_{L^n_{[t_0, t], x}}^n 
& = \int_{t_0}^{t}\int_{\T^4}
| I_N \Psi(x, t)|^n 
 dx dt 
 \leq n!  \int_0^T 
\int_{\T^4}
e^{| I_N \Psi(x, t)|}
 dx dt \\
& \le
n! e^{K T \log N} R \leq n^n e^{K T \log N} R.
\end{align*}

\noi
We now choose 
\begin{align*}
n \sim K T \log N +  c \log (100 N^\al ) \sim K T \log N \gg 1,
\end{align*}
where we may have to take $K\gg1$ larger.
\noi
Then, 
due to \eqref{K5a} and $\eta = n^{-1}$,
we can estimate the last term on the right-hand side of~\eqref{K5}
as
\begin{align} 
\begin{split}
& \Bigg\{\int_{t_0}^{t} \bigg(E^{1+c\eta}(t') + \frac{\eta}{(t'-t_0)^\frac12}\bigg) dt'\Bigg\}
\|I \Psi\|_{L^{\eta^{-1}}_{[t_0, t],x}}\\
& 
\leq 
\int_{t_0}^{t} \bigg(E(t') n e^{\frac 1 n (K T\log N + c \log (100N^\al )) } R^{\frac{1}{n}} + \frac{e^{\frac 1n K T \log N} R^\frac{1}{n}}{(t'-t_0)^\frac12}\bigg) dt'\\
& 
\les 
\int_{t_0}^{t} \bigg( \big[K R T \log N \big]E(t') + \frac{R}{(t'-t_0)^\frac12}\bigg) dt', 
\end{split}
\label{K8}
\end{align}

\noi
where we used that $R = R(\o) \geq 1$.

Next, we define $F$ by 
\begin{align}
 F(t) \deff \max_{t_0 \le \tau \le t}  E(\tau) - E(t_0) + N^\be \geq E(t).
 \label{K8a}
\end{align}

\noi
Then, by \eqref{K5a}, we have
\begin{align}
N^\be \leq F(t) \leq 200 N^\al
\label{K9}
\end{align}

\noi
for $t_0 \le t \le t_1$.
In particular, we have
$
\log F(t) \sim \log N.
$
%
%
Moreover, from \eqref{K9}, we have
 \begin{align} 
\begin{cases}
  N^{-2+3(2-s)} F^2(t) \les N^{-\al}F^2(t) \leq 200 F(t), \\
  N^{-\frac{1-2(2-s)}2 + \dl}F^\frac{3}{2}(t)
 \les N^{-\frac \al2 }F^\frac 32(t) \leq \sqrt{200} F(t),\\
  N^{-\frac{1-(2-s)}2 + \dl} F(t) \leq  F(t), \\
  N^\ld  F^{\frac34}(t) \les 
  N^\ld F^{-\frac{1}{4}}(t) F(t)
 \le F(t), 
\end{cases}
\label{K11}
\end{align}

\noi
provided that 
\begin{align}
\alpha \le \min(3s-4,2s-3), \quad 
\dl \le \min\big(\tfrac{2s - 3 - \al}{2},  \tfrac {s-1}2\big),
\quad \text{and}\quad  \ld \le \tfrac\be 4,
\label{K12}
\end{align} 
which requires that $s>\max(\frac43, \frac32)=\frac32$.
Hence, by \eqref{K8a}, \eqref{K5}, \eqref{K11}, \eqref{K8}, \eqref{K9}, and \eqref{K3}, we obtain 
\begin{align} 
\begin{split}
F(t)  & -    F(t_0) \\
& = \max_{t_0 \leq \tau \leq t} E (\tau) - E (t_0) \\
& \les
 (1 + M_T) \int_{t_0}^{t} F(t') dt' 
 + 
\int_{t_0}^{t} \bigg(K R T F(t') \log F(t')  + \frac{R}{(t'-t_0)^\frac12}\bigg) dt'\\
& \les
 (1 + V + K^3 T^3) \int_{t_0}^{t} F(t') dt' 
 + 
\int_{t_0}^{t} \bigg(K R T F(t') \log F(t')  + \frac{R}{(t'-t_0)^\frac12}\bigg) dt'\\
& \les
 (1 + V + K R T)
\int_{t_0}^{t} F(t') (\log F(t') + K^2 T^2) dt'   + 2 R (t - t_0)^{\frac 12}
\end{split}
\label{K13}
\end{align}

\noi
for any $t_0 \le t \le t_1$ such that \eqref{K9} holds.
Denoting by $C_0 = C_0 (s)$ the implicit constant in \eqref{K13}, 
we define $G$ by 
\begin{align}
 G(t) = F(t) - 2 C_0 R (t - t_0)^\frac{1}{2}.
 \label{K14}
\end{align}
Let us pick $t_*(s,R)>0$ such that 
\begin{align}
	2 C_0 R (t - t_0)^\frac{1}{2} \ll 1,
	\label{K15a}
\end{align} 
sufficiently small so that
\begin{align}
	F(t) \leq 5^{\frac{\al - \be}{2}} G(t).
	\label{K_add}
\end{align}
Then, from \eqref{K14} and \eqref{K_add}, we get that $F(t) \sim G(t)$, which combined with \eqref{K13} gives
\begin{align} 
\begin{split}
G(t)  -  G(t_0) 
\le C (1 + V + K R  T)
\int_{t_0}^{t}  G(t') (\log G(t') + K^2 T^2)  dt'
\end{split}
\label{K15}
\end{align}

\noi
for any $t_0 \le t \le \min( t_1, t_0 + t_*(s, R))$ and some $C>0$.

Now, note that the equation 
\[ \dt H(t) =\wt{C} H(t) (\log H(t) + K^2 T^2)\]

\noi
has a solution 
$ H(t) = \exp\big( \exp(\wt{C} t)  (\log H(0) + K^2T^2) - K^2T^2\big). $
%
Then, by comparison, we deduce from \eqref{K15} that 
\begin{align}
G(t) \leq  \exp \Big( e^{C(1 + V + K R T)  (t - t_0) } (\log G(t_0) + K^2T^2) - K^2T^2\Big)
\label{K16}
\end{align}

\noi
for some constant $C > 0$.

Recall from 
\eqref{K14} and 
\eqref{K8a} that 
$G(t_0) = N^\be$.
Then, if the condition
\begin{align}
e^{C(1 + V + K R T)  (t - t_0) } (\be \log N  + K^2 T^2) \le \al \log N +  K^2 T^2 - \frac{\al - \be}{2} \log 5
\label{K17}
\end{align}

\noi
holds for $t_0 \leq t \leq \min(t_1, t_0 + t_* (s, V, R, T, K))$ (where $t_* (s, V, R, T , K) > 0$ will be specified later), the bound \eqref{K16} implies 
\begin{align}
G(t) \le 5^{\frac{\be - \al}{2}}  N^\al
\label{K18}
\end{align}

\noi
for any $t_0 \le t \le \min( t_1, t_0 + t_*(s, V, R, T))$.
Then, we conclude from 
\eqref{K8a}, 
\eqref{K14}, \eqref{K15a}, and \eqref{K_add}
that 
\begin{align}
E(t) \le F(t) \le   N^\al
\label{K19}
\end{align}

\noi
for any $t_0 \le t \le \min( t_1, t_0 + t_*(s, V, R, T, K))$.
This in turn guarantees that the conditions~\eqref{K5a}
and~\eqref{K9} are met.
Therefore, 
by a standard continuity argument, 
we conclude that 
the bounds~\eqref{K18} and 
\eqref{K19} hold
for any $t$ with $t_0 \le t \le t_0 + t_*(s, V, R, T, K)$ such that 
the condition~\eqref{K17} holds.

Finally, let us consider the condition \eqref{K17}.
Let $\al = \al(s) > \be = \be(s) $ be such that the conditions in \eqref{K12} hold. Since $\al>\be$, there exists $t_{**}(s, V, R, T, K)$ such that, for $0\le  \tau\le t_{**}$, we have
\begin{align}
 \al  - e^{C(1 + V + R + KRT) \tau }\be \geq \frac{\al - \be}{2} > 0.
 \label{K20}
\end{align}

\noi
Then, since $N > 10$,
by choosing $0\le \tau \le t_{**}$ sufficiently small such that
\begin{align}
  e^{C(1 + V + R + KRT)  \tau } - 1
\leq \frac{\frac{\al - \be}{2} \log N - \frac{\al - \be}{2}\log 5}{K^2T^2}, 
 \label{K21}
\end{align}

\noi
we can guarantee that the condition \eqref{K17} is satisfied for $t_0 \le t \le t_0+\tau$, and hence so is \eqref{K19}.
This concludes the proof of Proposition \ref{PROP:LGWP}.
\end{proof}

We now present the proof of Theorem \ref{THM:GWP1}.
Fix  $\frac 74 < s <2$, $T\gg1 $, $\o \in \O$ such that $V = V(\o) < \infty$
and $R = R(\o)< \infty$, and let the parameters
$\al, \be,   \tau$ be as in Proposition~\ref{PROP:LGWP}.

Fix $N_0 \gg 1$ which is to be determined later.
Then, for $k \in \Z_{\geq 0}$, 
define an increasing sequence $\{N_k\}_{k \in \Z_{\ge 0}}$ by setting
\begin{align}
N_{k } =  N_0^{\s^{k}}
\label{Z2}
\end{align}

\noi
for some $\s = \s(s) > 1$ sufficiently large such that 
\begin{align}
N_{k+1}^{2(2-s)} N_k^\al + N_k^{2\al} \ll N_{k+1}^\be, 
\label{Z3}
\end{align}

\noi
which, due to to the assumptions on $\al>\be$, requires
\begin{align*}
2(2-s) < \be < \al \le 2s-3 \quad \text{and} \quad s>\tfrac74.
\end{align*}

\noi
Suppose that for some $k \in \Z_{\geq 0}$ and $t \geq 0 $, it holds that
\begin{align}
E(I_{N_k} \vec v)(t) \leq N^\al_k.
\label{Z4}
\end{align}

\noi
Then,  by \eqref{Hamil},
\eqref{I1}, Sobolev inequality, \eqref{Z4}, and \eqref{Z3}, 
we have 
\begin{align}
\begin{split}
E(I_{N_{k+1}}\vec v)(t) 
& \les N_{k+1}^{2(2-s)} \|\vec v\|_{\H^{s}}^2 + \|v\|_{H^{1}}^4 \\
& \les N_{k+1}^{2(2-s)} \|I_{N_k} \vec v\|_{\H^2}^2 + \|I_{N_k} v\|_{H^2}^4\\
&  \les  N_{k+1}^{2(2-s)}   E(I_{N_k}\vec v) + E(I_{N_k}\vec v)^2 \\
& \les  N_{k+1}^{2(2-s)} N_k^\al+ N_k^{2\al}
\ll N_{k+1}^\be.
\end{split}
\label{Z5}
\end{align}

We are now ready to implement  an iterative argument. 
Given $(u_0, u_1) \in \H^s(\T^4)$, choose $N_0 = N_0(u_0, u_1, s) \gg 1$
such that 
\begin{align} \label{eqn: sizeNk0}
E(I_{N_{0}}\vec v) (0) \leq N_{0}^\be.
\end{align}

\noi
By applying  Proposition \ref{PROP:LGWP}, there exists $\tau = \tau(s, T, K, \o)>0$ such that
\begin{align*}
E(I_{N_{0}}\vec v) (t) \leq N_{0}^\al
\end{align*}

\noi
for any $0 \le t \le  \tau$.
By \eqref{Z4} and \eqref{Z5}, this then implies 
\begin{align*}
E(I_{N_{1}}\vec v) (\tau)\leq N_{1}^\be.
\end{align*}

\noi
Applying   Proposition \ref{PROP:LGWP} once again, 
we in turn  obtain
\begin{align*}
E(I_{N_{1}}\vec v) (t) \leq N_{1}^\al
\end{align*}

\noi
for $0 \le t \le  2 \tau$.
By \eqref{Z4} and \eqref{Z5}, this then implies 
\begin{align*}
E(I_{N_{2}}\vec v) (2\tau) \leq N_{2}^\be.
\end{align*}

\noi
By iterating this argument $\big[\frac{T}{\tau}\big] + 1$ times, we 
obtain a solution $v$ to 
 the renormalized cubic SNLB \eqref{SNLB6} on
 the time interval $[0, T]$.
Since the choice of $T \gg1 $ was arbitrary, this proves global well-posedness of \eqref{SNLB6}.

\begin{remark} \label{REM:bound}\rm

From the argument above, we can also establish a growth bound on the Sobolev norm of the solution $v$ to SNLB \eqref{SNLB6}. Namely, for $T\gg1 $ and with the same choice of parameters, we have 
\begin{align*}
\| \vec v(t)\|_{\H^s} \les
\big( 1 + E(I_{N_{k}}\vec v) (t)\big)^\frac{1}{2} \leq N_{k}^\frac{\al}{2}
\end{align*}

\noi
for any $ 0\le t \le T $ such that 
 $k \tau \leq t \leq (k+1) \tau$ for some 
$k \in \Z_{\ge 0}$.
Then, by \eqref{Z2}, we have
\begin{align}
\| \vec v(t)\|_{\H^s} \les
\exp\big( \tfrac{\al}{2} \s^k \log  N_{0}\big)
\le
\exp\big( \tfrac{\al}{2}  \log  N_{0}
\cdot \exp \big(\tfrac{ (\log \s) t }{\tau}\big)\big)
\label{Z9}
\end{align}

\noi
for $0 \leq t \leq T$.
Moreover, 
in view of \eqref{eqn: sizeNk0}, we
choose $N_0 \in \NB$ such that 
$1 + E(I_{N_{0}}\vec v) (0) \sim N_{0}^\be$, so that by \eqref{I1} and the fact that $\be > 2 (2 - s)$, we have
\begin{align}
\log N_0 \sim \log 
\big(2 + \| \vec v(0)\|_{\H^s}\big).
\label{Z9a}
\end{align}

\noi
In order to iteratively apply Proposition \ref{PROP:LGWP} $ \frac{T}{\tau}$-many times to reach the target time~$T$, we need 
to guarantee the condition \eqref{K21}. By taking
\begin{align}
\tau \sim_{s, V, R, K} T^{-1},
\label{Z10}
\end{align}

\noi
 the condition \eqref{K20} holds. Thus, in view of \eqref{Z2} with $k  \sim \frac T\tau$, 
the condition \eqref{K21} becomes
\begin{align*}
  0 < C_0
\leq \frac{\frac{\al - \be}{2} \s^{T^2} \log N_0 - \frac{\al - \be}{2} \log 5}{K^2 T^2},
\end{align*}

\noi
which holds true
for any sufficiently large $T\gg1$.
Finally,  from 
\eqref{Z9}, \eqref{Z9a}, and \eqref{Z10}, 
we conclude the following double exponential bound for any $t\ge0$
\begin{align*}
\| \vec v(t)\|_{\H^s} 
\le
C \exp\Big( c \log 
\big(2 + \| \vec v(0)\|_{\H^s}\big)
\cdot e^{C(\o)  t^2}\Big).
\end{align*}

\end{remark}

\section{Almost sure global well-posedness of the hyperbolic $\Phi_4^{k + 1}$-model}
\label{SEC:GWP2}

In this section, we prove Theorem \ref{THM:2-old}, i.e.,  almost sure global well-posedness of the renormalized SdNLB \eqref{SNLB11} and invariance of the corresponding Gibbs measure \eqref{Gibbs3}. 
Due to the convergence of $\rhoo_N$ to $\rhoo$, 
the invariance of $\rhoo_N$ under 
the truncated SdNLB  dynamics~\eqref{SNLB10}, 
and Bourgain's invariant measure argument \cite{BO94,BO96},
Theorem \ref{THM:2-old}
follows once we can construct the limiting process $(u,\dt u)$ locally-in-time
with a good approximation property for
the solution $u_N$ to \eqref{SNLB10}.
Furthermore, since 
$\rhoo$ is mutually absolutely continuous with respect to $\muu_2$, 
it suffices to study  the renormalized SdNLB~\eqref{SNLB10} and \eqref{SNLB11} with the Gaussian 
random initial data $(u_0^\o, u_1^\o)$ 
with $\L(u_0^\o, u_1^\o) = \muu_2$.

We first detail how to adapt the proof of Theorem~\ref{THM:1} to show local well-posedness of \eqref{SNLB10} and \eqref{SNLB11}, uniformly in the truncation $N$, and then show invariance of the truncated Gibbs measure $\vec{\rho}_N$ in \eqref{GibbsN} under the dynamics of the truncated SdNLB \eqref{SNLB10}.

\smallskip

As in Section~\ref{SEC:LWP}, to construct solutions for SdNLB \eqref{SNLB10}-\eqref{SNLB11}, we proceed with a first order expansion centered around the stochastic convolution $\Psid$ which solves \eqref{SdLB}.
By defining the operator $\D(t)$ as 
\begin{equation*}
	\D(t) \deff e^{-\frac{t}2}\frac{\sin\big(t \jbb{\nb}^2\big)}{\jbb{\nb}^2} \qquad \text{with} \qquad \jbb{\nb} \deff \Big( \jb{\nb}^4 - \frac 14 \Big)^{1/4},
\end{equation*} 

\noi
the stochastic convolution $\Psid$ which solves the stochastic damped linear beam equation in \eqref{SdLB} can be expressed as 
\begin{align} 
	\Psid (t) 
	=\dt\D(t)u_0^\o + \D(t)(u_0^\o+u_1^\o)+ \sqrt{2}\int_0^t\D(t - t')dW(t'), 
	\label{PhiN}
\end{align}

\noi
where $W$ is a cylindrical Wiener process on $L^2 (\T^4)$ as in \eqref{Wiener1}. A direct but tedious computation shows that $\Psid_N=\pi_N\Psid$
is a mean-zero real-valued Gaussian random variable with variance
\begin{align*}
	\E \big[\Psid_N(t, x)^2\big] = \E\big[\big(\pi_N\Psid(t, x)\big)^2\big]
	= \al_N
\end{align*}

\noi
for any $t\in \R_+$, $x\in\T^4$,  and $N \in\NB$,
where $\al_N$ is as in \eqref{sN}.
Unlike $\s_N(t)$ in \eqref{sig1}, 
the variance $\al_N$ is independent of time $t$.
This is due to the fact that the massive Gaussian free field $\mu_2$
is invariant under the dynamics
of \eqref{SdLB}.

Let  $u_N$ be the solution to \eqref{SNLB10}
with $\L\big((u_N, \dt u_N)|_{t=0}\big) = \muu_2$.
Then, we write $u_N$ 
as 
\begin{align}
	u_N = v_N + \Psid
	= (v_N +  \Psid_N) + \pi_N^\perp \Psid,
	\label{decomp2}
\end{align}

\noi
where $\pi_N^\perp = \Id - \pi_N$. Note that the dynamics of 
the truncated Wick-ordered SdNLB~\eqref{SNLB10}
decouple into 
the linear dynamics for the high frequency part given by $\pi_N^\perp \Psid$
and 
the nonlinear dynamics for the low frequency part $\pi_N u_N$:
\begin{align}
	\dt^2 \pi_N u_N   + \dt \pi_N u_N  +(1-\Dl)^2  \pi_N u_N 
	+
	\pi_N\big(   \Wa{ (\pi_N u)^{k} } \big) 
	= \sqrt{2} \pi_N \xi  .
	\label{SNLB11a}
\end{align} 

	\noi
Then,   
by \eqref{Herm3-1}, the remainder term $v_N = \pi_N u_N - \Psid_N $ satisfies the following equation:
\begin{align}
	\begin{cases}
		\dt^2 v_N + \dt v_N +(1-\Dl)^2 v_N  +
		\sum\limits_{\ell=0}^k {k\choose \ell} \pi_N \big(     \Wa{(\Psid_N)^\l   }  v_N^{k-\ell}\big)
		=0,\\
		(v_N,\dt v_N)|_{t = 0}=(0,0),
	\end{cases}
	\label{SNLB12}
\end{align}

\noi
where the Wick power
$
	\Wa {(\Psid_N)^\l } 
	\, \deff  H_{\l} (\Psid_N; \al_N)
$
\noi
converges 
to a limit, denoted by 
$\Wa{ (\Psid)^\l  } \,$, 
in $C([0,T];W^{-\eps,\infty}(\T^4))$
for any $\eps > 0$ and $T > 0$, 
almost surely (and also in $L^p(\O)$
for any $p < \infty$); see Lemma \ref{LEM:sto_cov}.
Thus, we formally obtain the limiting equation:
\begin{align}
	\begin{cases}
		\dt^2 v + \dt v +(1-\Dl)^2v  +
		\sum\limits_{\ell=0}^k {k\choose \ell} \Wa{ (\Psid)^\l }  v^{k-\ell}
		=0,\\
		(v,\dt v)|_{t = 0}=(0,0).
	\end{cases}
	\label{SNLB13}
\end{align}

We now detail how to modify the proof of Theorem~\ref{THM:1} to show local well-posedness of \eqref{SNLB12}-\eqref{SNLB13}, uniformly in $N\in\NB$. Note that $v$ is a solution to \eqref{SNLB13} if and only if $w = e^{\frac{t}{2}} v$ satisfies the following equation:
\begin{align*}
	\dt^2 w + (1 - \Dl)^2 w - \frac 14 w + e^{\frac{t}{2}} \sum_{\l = 0}^k \binom{k}{\l} \Wa{(\Psid)^\l} (e^{-\frac{t}{2}} w)^{k - \l} = 0.
\end{align*}
The terms in the mild formulation corresponding to the $w$-equation can be treated as in Proposition \ref{PROP:LWPv}, except for the one coming from $\frac34 w - 2 \Dl w$ term. However, this term can be viewed as a perturbation thanks to the two degrees of smoothing in the integral Duhamel operator, and the analogue of Proposition~\ref{PROP:LWPv} follows. The same argument allows us to show local well-posedness of \eqref{SNLB12} where the time of existence depends only on the stochastic convolution $\Psid$ and its Wick-powers, but not on $N\in\NB$.

\smallskip

Now, it remains to show the invariance of 
the truncated Gibbs measure $\rhoo_N$ under 
the truncated SdNLB dynamics~\eqref{SNLB10} in the following proposition. In fact, the rest of the proof of Theorem \ref{THM:2-old}
follows from a standard application of Bourgain's invariant measure argument,
whose details we omit.
See, for example,~\cite{ORTz} for further details.

\begin{proposition}\label{LEM:global}
Let   $N \in \NB$.
Then, 
the truncated SdNLB equation \eqref{SNLB10}
is almost surely globally well-posed with respect to the random initial data distributed by the truncated Gibbs measure $\vec{\rho}_N$ in \eqref{GibbsN}.
Moreover, the truncated Gibbs measure $\rhoo_{ N}$  \eqref{GibbsN}
is invariant under the dynamics of 
\eqref{SNLB10}. More precisely, denoting by $u_N$ the global solution to truncated SdNLB equation \eqref{SNLB10}, we have $\L (u_N (t), \dt u_N (t)) = \vec{\rho}_N$ for any $t \in \R_+$.
\end{proposition}

\begin{proof}
The idea of the proof has already appeared in \cite{GKOT, LO22, ORTz, OOT} and so we only sketch the key steps.
Given $N \in \NB$, we define $\vec{\mu}_{2, N}$ and $\vec{\mu}^\perp_{2,N}$ to be the marginal probability measures on $\pi_N \H^{-\eps} (\T^4)$ and $\pi_N^\perp \H^{-\eps} (\T^4)$, respectively. In other words, recalling $X^1$ and $X^2$ in \eqref{series}, $\vec{\mu}_{2, N}$ and $\vec{\mu}^\perp_{2,N}$ are the induced probability measures under the maps $\o\in\O\mapsto(\pi_N X^1 (\o), \pi_N X^2 (\o))$ and $\o\in\O\mapsto (\pi_N^\perp X^1 (\o), \pi_N^\perp X^2 (\o))$, respectively.
%
%
%
Then, with $\vec{\mu}_2 = \vec{\mu}_{2, N} \otimes \vec{\mu}_{2, N}^\perp$ and \eqref{GibbsN}, we can write
\begin{align}
\vec{\rho}_N = \vec{\nu}_N \otimes \vec{\mu}_{2, N}^\perp,
\label{rho_decomp}
\end{align}

\noi
where the measure $\vec{\nu}_N$ is given by
\begin{align*}
d \vec{\nu}_N = Z_N^{-1} R_N (u) d \vec{\mu}_{2, N}
\end{align*}

\noi
with the density $R_N$ as in \eqref{R1}.

We recall the decomposition \eqref{decomp2}. Since the high frequency part $\pi_N^\perp u_N = \pi_N^\perp \Psid$ satisfies
\begin{align}
\dt^2 \pi_N^\perp \Psid + \dt \pi_N^\perp \Psid + (1 - \Dl)^2 \pi_N^\perp \Psid = \sqrt{2} \pi_N^\perp \xi,
\label{high}
\end{align}

\noi
the dynamics of $\pi_N^\perp\Psid$ are linear and thus we can study the evolution of each frequency on the Fourier side to conclude that the Gaussian measure $\vec{\mu}_{2, N}^\perp$ is invariant under the dynamics of \eqref{high}. In fact, a tedious but direct computation shows that 
$$\E \big[ |\ft \Psid (t, n) |^2\big] = \frac{1}{\jb{n}^4} \quad \text{and} \quad \E \big[ |\ft{\dt \Psid} (t, n)|^2 \big] = 1$$ 
for any $t \in \R_+$ and $n \in \Z^4$, so that $\L (\Psid (t), \dt \Psid (t)) = \vec{\mu}_{2}$ for any $t \in \R_+$.

We now consider the low frequency part $\pi_N u_N$, which solves \eqref{SNLB11a}. Denoting $(u_{1, N}, u_{2, N}) = (\pi_N u_N, \dt \pi_N u_N)$, we can write \eqref{SNLB11a} in the following Ito formulation:
\begin{multline}
d \begin{pmatrix} u_{1, N} \\ u_{2, N} \end{pmatrix} 
+ \Bigg\{ \begin{pmatrix} 0 & -1 \\ (1 - \Dl)^2 & 0 \end{pmatrix}
\begin{pmatrix} u_{1, N} \\ u_{2, N} \end{pmatrix}
+ \begin{pmatrix} 0 \\ \pi_N  \Wa{u_{1, N}^k} \end{pmatrix} \Bigg\} dt \\
= \begin{pmatrix} 0 \\ -u_{2, N} dt + \sqrt{2} \pi_N dW \end{pmatrix}.
\label{SNLB16}
\end{multline}

\noi
This shows that the generator $\L^N$ of the Markov semigroup for  \eqref{SNLB16}
 can be written as $\L^N = \L^N_1 + \L^N_2$, 
 where $\L_1^N$ denotes corresponds to the (deterministic) NLB with truncated nonlinearity
 \begin{align}
 	\begin{split}
 		d  \begin{pmatrix}
 			u_{1,N} \\ u_{2,N}
 		\end{pmatrix}
 		+  \Bigg\{
 		\begin{pmatrix}
 			0  & -1\\
 			(1-\Dl)^2 &  0
 		\end{pmatrix}
 		\begin{pmatrix}
 			u_{1,N} \\ u_{2,N}
 		\end{pmatrix}
 		+  
 		\begin{pmatrix}
 			0 \\ \pi_N\Wa{ u_{1,N}^k  } 
 		\end{pmatrix}
 		\Bigg\} dt = 0, 
 	\end{split}
 	\label{SNLB17}
 \end{align}
 while $\L_2^N$ corresponds to the Ornstein-Uhlenbeck process: 
\begin{align}
d   u_{2,N}
   = 
 -  u_{2,N} dt + \sqrt 2\pi_N   dW.
\label{SNLB18}
\end{align}

\noi
The invariance of 
 $ \vec{\nu}_N $ under the dynamics of \eqref{SNLB17}
 follows from Liouville's theorem and the conservation of the Hamiltonian $E_N(u_{1,N}, u_{2,N})$ under the dynamics of 
  \eqref{SNLB17}, where
\begin{align*}
E_N (u_{1,N}, u_{2,N} ) = \frac{1}{2}\int_{\T^4}  |  (1-\Dl) u_{1,N} |^2  dx +  
\frac{1}{2}\int_{\T^4} (u_{2,N})^2dx
+   \frac{1}{k+1} \int \Wa{u_{1,N}^{k+1}}\,dx.
\end{align*}
 
\noi
Hence, we have 
 $(\L^N_1)^*    \vec{\nu}_N  = 0$, where $(\L^N_1)^*$ denotes the adjoint of $\L^N_1$. Regarding \eqref{SNLB18}, 
 we recall that the Ornstein-Uhlenbeck process preserves the standard Gaussian measure.
 Thus, $ \vec{\nu}_N $
 is also invariant under the dynamics of~\eqref{SNLB18}, 
 since the measure $\vec{\nu}_N$ on the second component
 is  the white noise $\mu_0$ 
 (see \eqref{gauss0} with $s=0$ and projected onto the low frequencies ${|n|\leq N}$). 
Hence, we have $(\L^N_2)^*    \vec{\nu}_N  = 0$, and so
\[(\L^N)^*    \vec{\nu}_N = 
(\L^N_1)^*    \vec{\nu}_N +  (\L^N_2)^*    \vec{\nu}_N  = 0.\]

\noi
This shows invariance of $\vec{\nu}_N$ under \eqref{SNLB16}
and hence under \eqref{SNLB11a}.

Therefore, 
invariance of 
 the truncated Gibbs measure 
$\vec{\rho}_N$ in \eqref{GibbsN}
under the  truncated SdNLB dynamics \eqref{SNLB10}
follows from 
\eqref{rho_decomp}
and the invariance of   $\vec{\nu}_N$ and $\vec{\mu}_{2, N}^\perp$
under
\eqref{SNLB16} and 
 \eqref{high}  respectively.
\end{proof}

\begin{ack}\rm
The authors would like to thank Tadahiro Oh for proposing the problem and for his support during the project. The authors also thank Chenmin Sun for helpful discussions. A.C., G.L., and R.L.~were supported by the European Research Council (grant no.~864138 ``SingStochDispDyn'').
G.L.~was also supported by the EPSRC
New Investigator Award (grant no.~EP/S033157/1).
\end{ack}

%


\end{document}